\documentclass[11pt]{article}
\usepackage{amssymb}
\usepackage{amsmath}
\usepackage{amsfonts}
\usepackage{graphicx}
\usepackage{color}
\usepackage[round]{natbib}
\usepackage{latexsym}
\usepackage{float}

\newtheorem{theorem} {\noindent \bf{Theorem}}

\newtheorem{lemma}{\noindent \bf{Lemma}}
\newtheorem{proposition}{\noindent \bf{Proposition}}

\newenvironment{proof}[1][Proof]{\noindent \textbf{#1.} }{\ \rule{0.5em}{0.5em}}
\newcounter{remno} 
\newcommand\aser{a^s}
\newcommand\aad{a^d}
\newcommand\piser{\pi^s}
\newcommand\piad{\pi^d}
\newcommand\Aser{A^s}
\newcommand\Aad{A^d}
\newcommand\aadh{\hat{a}^d}
\newcommand\act{a}
\newcommand\Aset{{\cal A}}
\newcommand\Ex{{\bf{E}}}
\begin{document}
\title{The Value of Service Rate Flexibility in an $M/M/1$ Queue with
  Admission Control
\footnote{Research supported by the Greek Secretariat of Research and Technology, via a Greece-Turkey bilateral research program. }
}
\author{Yiannis Dimitrakopoulos\\
              Department of Informatics \\
              Athens University of Economics and Business \\
              76 Patission Str., 10434 Athens, Greece \\
              {dimgiannhs@aueb.gr}           
         \and
        Apostolos Burnetas\\
                Department of Mathematics\\
                University of Athens\\
                Panepistemiopolis, 157 84 Athens, Greece\\
                {aburnetas@math.uoa.gr}
}
\date{\today}

\maketitle

\begin{abstract}
We consider a single server queueing system with admission control and the possibility to switch dynamically between $k$ increasing service rate values with service cost rate being convex in service rate. We explore the benefit due to service rate flexibility on the optimal profit and the admission thresholds, when service payment is made upon customer's admission. We formulate a Markov Decision Process model for the problem of joint admission and service control considering both discounted and average expected profit maximization, and show that the optimal policy has a threshold structure for both controls. Regarding the benefit due to flexibility, we show that it is increasing in system congestion, and that its effect on the admission policy is to increase the admission threshold. We also derive a simple approximate condition between the admission reward and the relative cost of service rate increase, so that the service rate flexibility is beneficial. We finally show that the results extend to the corresponding model where service payment is made at the end of each service completion and differences on the benefit due to service flexibility with respect to the original model are pursued numerically.
\end{abstract}

\section{Introduction}
\label{sec-Introduction}

Admission control is a queue management tool that can increase the
efficiency of resource utilization in many service systems. Depending
on the particular application, it can be used to preserve system
capacity for future customers who bring higher profit, to limit the
number of admitted customers in order to provide a better quality of
service to those already in, etc. Admission control is often employed
indirectly via dynamic pricing, or by price discrimination (direct or
indirect) among different customer types. However in several
situations frequent price changes may not be feasible, and in order to
contain congestion, denying service to certain arriving customers may
be necessary.

On the other hand, in several queueing systems it is possible to
alleviate the congestion effects by adjusting the service
capacity. For example in banks or call centers the number of servers
may change during the day to follow variations of the arrival
rate. The service rate may also be varied dynamically when the queue
length becomes too long. Increasing (and in some cases decreasing) the
service rate comes at a cost, but it has the advantage over admission
control policies that fewer or no customers are turned away. It is
thus of interest to explore to what extent a flexibility in service
capacity can interact with admission control and when it can alleviate
its effects.

In this paper we investigate the interaction of service rate
flexibility and admission control in an M/M/1 queueing system, where the
flexibility is modeled as the option to switch the service rate dynamically based on the congestion. We analyze the problem of joint admission
and service rate control by formulating a Markovian Decision Process
model to maximize the infinite horizon expected discounted and average profits.
We explore the effect of service rate flexibility on the optimal
profit and the admission thresholds.  Specifically, the benefit of the
service rate switch option is compared against a baseline case where
only the lowest service rate is used and admission control is employed.
Thus, in principle, the benefit is due both to the existence of
higher service rates as well as the flexibility to use them. On the other
hand, higher service rates are not free but come at higher costs,
thus the benefit of using them reflects the tradeoff between serving at a
faster rate and paying higher operational costs. When we use the term
``benefit of flexibility'' we mean exactly how this tradeoff manifests
itself in the presence of admission control.

Specifically, we consider a simple admission-service control model
with a finite number of available service rates, a fixed service reward per customer, holding costs
that are increasing and convex in the number of customers and service cost rates that are increasing and convex
in the service rate. We first show that the optimal policy has a
threshold structure for both service and admission controls. We define the value of flexibility as the expected profit increase in the optimal
admission control subproblem when the dynamic service rate change
option becomes available, and show that this value is nondecreasing with
the level of congestion. Furthermore, the effect on the
optimal policy is to increase the admission threshold. We derive
a simple sufficient condition between the admission reward and the service costs
under which the service flexibility brings no benefit. Finally, we show that
these results are valid when the service reward is obtained before or after service completion.

The single server model studied in this paper is more relevant in a make-to-order production setting, as opposed to a service provider system
in which the service rate adjustment is made by varying the number of servers. A make-to-order production firm may be forced to apply admission control to incoming orders, if it has to pay high penalties for delayed deliveries. In such cases, if there exist a possibility to increase the production rate at times of high congestion, the need for order rejections may be alleviated. Depending on the structure of the production process, the capacity increase could be achieved either by increasing the machine load, if the production equipment has speed flexibility, or by increasing the number of shifts in a day.

Motivated by such a framework, the assumptions of convex holding cost is realistic because it reflects constant or increasing marginal penalties for order delays. It is also reasonable to assume service cost rates that are convex in the service rate, because maintenance costs may increase disproportionately with a high production speed, or because additional shifts require higher overtime payment rates.

The paper develops a model of dynamic optimization of queueing
systems, a large area with very extensive literature. Both admission
and service control models have been studied thoroughly.
\cite{stidhamsurvey} and \cite{walrand} survey several dynamic
optimization models developed for queueing control.

For models with a single class of customers, as the one analyzed in this paper, admission
control makes sense when there is an exogenous holding cost rate
function. A simple model in this direction was first presented in
\cite{naor}, where arriving customers are admitted or not based on the
observed queue length with the objective to maximize the overall customer 
benefit from receiving a reward for service completion minus a linear increasing holding cost per unit time of
delay. It is shown that a socially optimal policy admits fewer
customers than those who would decide to enter based on an individual
optimality criterion. \cite{stidham1985} considers a $GI/M/1$ queue
under infinite horizon discounted cost, assuming a convex and
nondecreasing holding cost rate function. It is shown that the
optimal policy has a threshold structure if and only if the optimal
benefit is concave in the number of customers in the system which in
turn depends on convexity of the holding cost rate function. As in
\cite{stidham1985}, we also consider a convex, nondecreasing holding
cost rate which implies a threshold property of the optimal admission
policy.

On the other hand, in multi-class systems with finite capacity, 
admission control may be useful even in the absence of holding costs,
because in this case admitting a customer implies the possibility of a
loss of profit from a future higher class customer. \cite{miller}
considers a system with $n$ parallel and identical servers, no waiting
room and $m$ customer classes which contribute to the system
different fixed rewards. This model results in a threshold-type
optimal policy with a preferred class. \cite{lippross} analyze the
optimal admission rule for a system with one server and no waiting
room which receives offers from customers according to a joint service
time and reward probability distribution. \cite{munoz}
and \cite{ormeciburnetas2} also investigate properties of
optimal admission policies for certain loss systems. \cite{munoz}
develop an optimal static admission policy in an $M/G/c/c$
queueing system with $k$ customer classes with generally different service requirements and service rewards. \cite{ormeciburnetas2} examine the problem of dynamic admission control in a two class loss Markovian queueing system with
different service rates and different fixed rewards for the two
customer classes.

The admission control problem has been also analyzed in queueing
systems under heavy traffic. In this framework,
the dynamic optimization is usually approximated by a diffusion
control problem, following the approach of \cite{harrison1988}.
Recent works in this area include \cite{ward_kumar2008} and
\cite{kocaga_ward2010}, both analyzing admission control under
customer abandonments. \cite{ward_kumar2008} analyze a $GI/G/1$ queue
in the conventional heavy traffic regime, where the optimal control
depends on the sample path of the diffusion and the resulting
asymptotically admission control policy of threshold type depends on
second moment data of the interarrival and service
times. On the other hand, \cite{kocaga_ward2010} consider the long run average
cost minimization problem of a multi-server system
with a single arrival and a single server class in the Halfin-Whitt heavy traffic regime.

Dynamic service control in queueing systems is an equally large
field. Several problems can be viewed as service control models,
including controlled server vacations, server allocation policies in
polling systems, etc. In an early work, \cite{crabill1974} examines
dynamic service control under infinite horizon expected average
expected cost in a maintenance system with finite available service
rates, a linear holding cost rate and a reward collected in service
completions. It is shown that the optimal service rate is increasing
in the number of customers waiting in line. The monotonicity of the
optimal service rate is also shown in \cite{lippman1975a} in the
framework of an $M/M/1$ queue, with service rates varying in a closed
set and the holding cost rate increasing and convex.
\cite{georgeharrison} consider the service control problem in an
$M/M/1$ queue where service rates are dynamically selected from a
close subset of $[0,\infty]$, under no switching cost, state-dependent
holding cost and rate dependent service cost. They develop an
asymptotic method for computing the optimal policy under average cost
minimization by solving a sequence of approximating problems, each
involving a truncation of the holding cost function. They prove that
the optimal policies of the approximating problems converge
monotonically to the optimal policy of the original problem and derive
an implementable policy and a performance bound at each iteration. In
our model we also derive a monotonicity property of the service
control component of the problem, under a convexity assumption on the
holding cost function. In the works mentioned above there are no
switching costs for changing the service rate. We refer to
\cite{luserfozo}, \cite{Hipp} and \cite{Kitaevserfozo} for models that
include service rate switching costs, resulting in hysteretic
policies.

There are several works that consider and analyze the joint control problem of admitting or rejecting an incoming customer and
varying the service rate. Most of them are motivated by and extend the work of \cite{georgeharrison}.
For instance, in the area of asymptotic analysis in the heavy traffic regime considering diffusion control problems we refer to \cite{ghosh_weerasinghe2007} and
\cite{ghosh_weerasinghe2010}. \cite{ghosh_weerasinghe2007} examine a
queueing network where a central planner dynamically selects the
service rate and buffer size that minimize the long-run average
expected cost. The optimal policy is derived from the solution of a
Brownian control problem and it consists of a feedback-type drift
control and a threshold type admission policy.
\cite{ghosh_weerasinghe2010} consider a Markovian system with customer
abandonments and address the infinite horizon discounted problem.
In contrast to  \cite{ghosh_weerasinghe2007}, it is  proved that the optimal
joint dynamic policy derived from the solution of the Brownian control
problem is asymptotically optimal for the original problem. A common feature of these papers as well as the admission control problem in \cite{kocaga_ward2010} is that uniformization is not applicable, because the transition rates are generally unbounded.

More relevant to our work are \cite{ata} and \cite{hasenbein2010}, since they also consider
the joint control problem in a simple setting of an $M/M/1$ queue under an average reward/cost criterion.
More specifically, \cite{ata} consider the
joint admission and service control problem in an M/M/1 queue with
adjustable arrival and service rates, under long-run average welfare
maximization. They also formulate and solve an associated dynamic
pricing problem. They show that the optimal arrival and service rates
are monotone in the system length. However the optimal prices, which
are set to induce the optimal arrival and service rates, are not
necessarily monotone. Finally, they find that dynamic policies can
result in significantly higher profits compared to static policies.
Similarly to \cite{ata}, \cite{hasenbein2010} develop an efficient
iterative method for computing the optimal policy under an average
cost criterion, providing a computable upper bound on the optimality
gap at each iteration step. It is also shown that service rates are
monotone increasing in the system state.

Although the model considered by \cite{hasenbein2010} can be seen as more general than ours since they assume a continuum action set for service decisions, our simpler setting contributes to the research area in several ways. First, we analyze the discounted expected profit maximization in the infinite horizon
and show that the optimal policy converges to the corresponding long run average optimal policy, under easily verifiable sufficient conditions on the cost functions. Moreover, we characterize the optimal policy as threshold-based and derive a simple intuitive condition which orders the admission and service thresholds. This ordering is helpful, because it allows identifying which of the available service rates are not useful. The question on which of the service rates are actually beneficial in a given problem motivated us to extend the discussion on \cite{ata} and \cite{hasenbein2010} by introducing the service rate flexibility, in order to examine the effect of the switch option on the admission policy and the optimal profit.

The rest of the paper is organized as follows. In Section
\ref{sec-Model} we define the joint control model, show several
properties of the value function and establish the threshold structure
of the optimal policy.  In Section \ref{sec-VoSF} we analyze the value
of service flexibility and the effect of the high service rate switch
option on the admission policy. In Section \ref{sec-Average} we
analyze the problem under the average reward criterion.  In Section
\ref{Rdepart} we consider a variation of the original model, in which
the service reward is collected at departure epochs and show that the
main results still hold.  In Section \ref{sec-Numerical} we present a
set of computational experiments exploring the sensitivity in the
system parameters. Section \ref{sec-Conclusions} concludes.

\section{Model Description}
\label{sec-Model}
We consider a single server Markovian
queue under the FCFS discipline, where customers arrive according to
a Poisson process with rate $\lambda$. The service rate may be
dynamically switched among $k$ available values $\mu_1<\mu_2<\dots<\mu_k$ without any cost.
The service provider receives a fixed
reward $R$ per customer admitted, and incurs holding and
service costs as follows. The holding cost is equal to $h(x)$ per
unit time, where $x$ is the number of customers in the system. The
function $h(x)$ is assumed to be increasing and convex. The service cost
is equal to $c_j$ per unit time, for using the service rate
$\mu_j$, respectively, for $j=1,\ldots,k$. We assume that
\begin{equation}\label{convcost} \frac{c_j-c_{j-1}}{\mu_j-\mu_{j-1}}\leq\frac{c_{j+1}-c_j}{\mu_{j+1}-\mu_j},\;\hbox{for }j=2,\ldots,k-1,\end{equation}
which corresponds to the service cost rate being convex in the service rate. We let $\zeta_j=\frac{c_j-c_{j-1}}{\mu_j-\mu_{j-1}}$ for $j=2,\ldots,k$, with $\zeta_1=0$ and $\zeta_{k+1}=+\infty$ and the convexity assumption for service cost rate can be rewritten as $\zeta_1\leq\zeta_2\leq\ldots\zeta_{k+1}$. Finally, we assume that the server is not allowed to close down when the system is empty. It
is obvious that in this state the optimal service rate is the lowest, and, thus, cost rate $c_1$ is incurred.

Since the system is Markovian, it suffices to assume that the system
manager makes a decision at both arrival and departure epochs. Service
rate decisions can be made at both arrival and departure epochs,
whereas admission decisions are made only at arrival epochs. Assuming
continuous time discounting at rate $\beta>0$, the service provider's
objective is to maximize the infinite horizon expected discounted net
profit. Thus, the problem can be framed as a continuous time Markov
Decision Process as follows.

Let $T_j$ be the time of the $j^{th}$ arrival, $X(t)$ a random
variable denoting the number of customers in the system at time $t$
and $I(t)={\bf{1}}(t=T_j\;\mbox{for some}\;j)$ the indicator of the
event that $t$ is an arrival epoch. We define the state vector as the
pair $(X(t),I(t))$, thus the state space is $S=\mbox{I\!N}_0 \times
\{0,1\}$. State $(0,0)$ denotes an empty system.

For the action sets, let $\Aser(t) \in \{1,2,\dots,j,\ldots,k\}$ denote the service
rate employed at time $t$, where $j$ stands for service rate $\mu_j$, respectively, and $\Aad(T_j)\in \{0,1\}$ the admission decision at the $j^{\hbox{th}}$
arrival epoch, where $0,1$ denote rejection and admission, respectively.
 In states $(X(t),I(t))=(x,1)$ corresponding to arrival
epochs, the action is defined by the pair $a(t)=(\aad(t),\aser(t))$,
thus the action set is $\Aset(x,1)=\{0,1\}\times\{1,2,\ldots,k\}$. On the other
hand, in states $(X(t),I(t))=(x,0)$ corresponding to departure
epochs, the action is defined only by $\aser(t)$, thus the action
set is $\Aset(x,0)=\{1,2,\ldots,k\}$.
Finally, let $\Pi$ be the space of history dependent
policies and $\upsilon^{\pi}_{\beta}(x,i)$ denote the infinite
horizon expected $\beta-$discounted net profit with initial state
$(x,i)\in S$
\begin{eqnarray*}
  \lefteqn{\upsilon^{\pi}_{\beta}(x,i)=} \\
  & & \Ex^{\pi} \left [ \sum_{j=0}^{\infty}e^{-\beta T_j} R\ {\bf{1}}(\Aad(T_j)=1) \right .\\
  & & \left . - \int_{0}^{\infty}e^{-\beta t}[h(X(t))+c(\Aser(t))]dt | X(0)=x, I(0)=i \vphantom{\sum_{j=0}^{\infty}} \right ].
\end{eqnarray*}
The optimal value function is
\begin{equation}\label{CTMDP}
\upsilon_{\beta}(x,i)=\sup_{\pi\in\Pi}\upsilon^{\pi}_{\beta}(x,i),
\end{equation}
and a policy $\pi^*$ is optimal if $\upsilon^{\pi^*}_{\beta} = \upsilon_{\beta}$.

\subsection{An Equivalent Model in Discrete Time}\label{subsec-DTMDP}
We can construct a discrete-time version of the Markovian Decision
problem as follows. Depending on the state and the action employed,
the transition rate out of any state can take values $\lambda$,
or $\lambda+\mu_j$, for some $j=1,\ldots,k$. Let $\Lambda=\lambda+\mu_k$
denote the maximum transition rate out of any state. Using standard
uniformization arguments (see Section 11.5 of \cite{puterman}), it
follows that the model described in \eqref{CTMDP} is equivalent to a
model where the transition rates are all equal to $\Lambda$ and the
transition probabilities are appropriately modified.  Since in the
original continuous time model the transition rates out of a state are
generally different in different states, the discrete time formulation
allows for transitions from a state back to itself so that the expected sojourn times are equal in the two models. These are referred to as fictitious transitions.

Via this transformation, the problem can be written in a form
equivalent to a discrete time discounted Markov Decision Process, as
follows
\begin{eqnarray}
  \upsilon(x,1)&=& \max \left \{ R + \upsilon(x+1,0), \upsilon(x,0)\right \},\;x\geq0\label{DTMDPu1} \\
  \nonumber\upsilon(x,0)&=&\frac{\Lambda}{\Lambda+\beta}
  \left
  \{ -\frac{h(x)}{\Lambda}
  + \frac{\lambda}{\Lambda} \upsilon(x,1)
  \right.
  \\
  && \left. + \max_{j=1,\ldots,k} \{ -\frac{c_j}{\Lambda}+\frac{\mu_j}{\Lambda}\upsilon(x-1,0)+
  \frac{\mu_k-\mu_j} {\Lambda}\upsilon(x,0)\}
  \right
  \},\;x>0\label{DTMDPu2}\\
  \upsilon(0,0)&=&\frac{\Lambda}{\Lambda+\beta} \left \{-\frac{h(0)}{\Lambda}-\frac{c_1}{\Lambda}+ \frac{\lambda}{\Lambda}\upsilon(0,1)+\frac{\mu_k}{\Lambda}\upsilon(0,0) \right \},\label{DTMDPu3}
\end{eqnarray}
where for simplicity we omit the subscript
$\beta$.

Note that in the discrete time formulation, the equivalent discount
factor per transition is equal to $\alpha=\frac{\Lambda}{\Lambda+\beta}$ and for $\beta >0$ it has the
standard property $0 < \alpha <1$. For ease of the exposition, we
normalize the time scale so that $\Lambda+\beta=1$. This normalization is without loss of generality. Indeed, if $\Lambda+\beta\neq1$ we can make the following transformation on the parameters:
$\tilde{\lambda}=\frac{\lambda}{\Lambda+\beta},\;\tilde{\mu}_j=\frac{\mu_j}{\Lambda+\beta}$ and $\tilde{c}_j=\frac{c_j}{\Lambda+\beta}$, for $j=1,\ldots,k$, $\tilde{h}(x)=\frac{h(x)}{\Lambda+\beta}$, $\tilde{R}=R$ and, finally, $\tilde{\beta}=\frac{\beta}{\Lambda+\beta}$. Under this transformation, it follows that $\tilde{\Lambda}+\tilde{\beta}=1$, which implies that the value function and, thus the optimal policy derived by the system of optimality equations in \eqref{DTMDPu1} - \eqref{DTMDPu3} are identical with and without the normalization. Note that in Section \ref{sec-VoSF} where we consider the criterion of
average reward per unit-time as a limit of the discounted reward
problem when $\beta \to 0$ and thus $\alpha \to 1$, we do not make
this normalization assumption.

The finite horizon version of this last model is the
following, where $\upsilon_n(x,i)$ denotes the optimal discounted
profit for the remaining $n$ transitions starting at state $(x,i)$.
\begin{eqnarray}
\upsilon_{n+1}(x,1)&=& \max\{R+\upsilon_{n+1}(x+1,0),\upsilon_{n+1}(x,0)\},\;x\geq0,\label{DTMDP1} \\
\nonumber\upsilon_{n+1}(x,0)&=&-h(x)+\lambda\upsilon_{n}(x,1)\\
&&+\max_{j=1,\ldots,k}\{-c_j+\mu_j\upsilon_{n}(x-1,0)+(\mu_k-\mu_j)\upsilon_{n}(x,0)\},\;x>0,\label{DTMDP2}\\
\upsilon_{n+1}(0,0)&=&-h(0)-c_1+\lambda\upsilon_{n}(0,1)+\mu_k\upsilon_{n}(0,0),\label{DTMDP3}\\
\upsilon_{0}(x,i)&=&0,\;x\geq0,\;i\in\{0,1\}.\label{DTMDP-int}
\end{eqnarray}
Note that in \eqref{DTMDP1} the iteration index on the right hand side is still $n+1$, because after an admission decision in state $(x,1)$ there is an instantaneous state switch to state $(x+1,0)$ or $(x,0)$, so that the corresponding service-rate decision can also be made at that instant.
The advantage of writing the optimality equations in this form is that only admission decisions are made in states $(x,1)$ and only service decisions in states $(x,0)$.

Since the state space is infinite and the one-step reward function is
not necessarily bounded, the convergence of
\eqref{DTMDP1}-\eqref{DTMDP-int} to the optimal value function must be
established.

To this end, we make the following assumption ensuring that the
holding cost does not increase too rapidly with the queue length.
\ \\[12pt]
\noindent{\bf Assumption 1}\ \\[-0.6cm]
\begin{enumerate}
\item[(c1)] There exists a constant $\theta>1$ such that: $h(x+1)\leq\theta h(x),\;\hbox{for any}\;x > 0.$
\item[(c2)] There exists a constant $\alpha\in[0,1)$ and a positive
integer $J$ such that: for $x\geq0$,
\begin{equation*}\label{VerThc2}\Lambda^J[R+c_k+h(x+J)]\leq\alpha[R+c_k+h(x)].\end{equation*}
\end{enumerate}

Assumption 1 is quite general.  It can be easily seen
that it is satisfied for power cost functions $h(x)=Kx^m,\;K>0,m
\geq 1$, as well as, exponential cost functions $h(x)=K\rho^x$ with
$K>0$ and $\rho\in(1,\frac{1}{\Lambda})$.

In the next theorem we show that under Assumption 1 there exists an
optimal policy for the discounted problem and the finite horizon
approximations converge to the unique solution of
\eqref{DTMDPu1}-\eqref{DTMDPu3}.

\begin{theorem}\label{VerificationTheorem}
If the holding cost rate function $h(x)$ satisfies Assumption 1, then
\begin{itemize}
\item[{\bf{i.}}] The system of equations \eqref{DTMDPu1}-\eqref{DTMDPu3} has a unique solution, which equals $\upsilon^*_{\beta}$.
\item[{\bf{ii.}}] There exists a stationary deterministic optimal policy.
\item[{\bf{iii.}}] The solution of the system of equations \eqref{DTMDP1}-\eqref{DTMDP-int} converges to $\upsilon^*_{\beta}$.
\end{itemize}
\end{theorem}

\begin{proof}
The proof follows by applying Theorem 11.5.3. of
\cite{puterman}. To do this we must verify the following
\begin{itemize}
\item[1.] Assumption 11.5.1 (\cite{puterman}) implies that all
transitions rates are bounded above. This is satisfied here with
$\Lambda$ being the upper bound.
\item[2.] There exists a function $w:S\rightarrow{\mbox{I\!R}}$ such that
\begin{itemize}
\item[2a.]
$\max_a|r(s,a)|\leq Mw(s),\;\forall s\in S$, where $r(s,a)$ is the one period profit function in the discrete time MDP \eqref{DTMDP1}-\eqref{DTMDP-int} and $M$ is a constant.
\item[2b.]There exists a non-negative constant $k<\infty$ for which
\begin{equation*}
\Ex^{\pi}\{w(X_{n+1})|X_{n}=x,I_n=i,Y_n=a\}\leq kw(x,i),
\end{equation*}
for all $a\in \Aset(x,i)$ and $(x,i)\in S$.
\item[2c.]There exists constant $\alpha\in[0,1)$ and $J\in\bf{Z}$ such that
\begin{equation*}
\Lambda^J  \Ex^{\pi}\{w(X_{n+J},I_{n+J})|X_{n}=x,I_{n}=i\}\leq\alpha w(x,i).
\end{equation*}
\end{itemize}
\end{itemize}
We will verify that the function
$w(x,i) \equiv w(x) = R+c_k+h(x),\;(x,i)\in S$ satisfies the above properties.
Note that $w(x)$ is increasing in $x$.

To show 2a, note that
\begin{equation*}
r(s,a)=\left\{\begin{array}{ll} R-c_j-h(x),&\hbox{for}\;s=(x,1),\;a=(1,j)\\
                               -c_j-h(x),&\hbox{for}\;s=(x,i),\;a=(0,j)\hbox{ or }a=j\\
                            \end{array}\right.
                            \end{equation*}
Thus, 2a holds for $M=1$.

To show 2b, for any $(x,i)\in S$ and $a\in \Aset(x,i)$ the possible transitions are to states with $x-1, x$ or $x+1$ customers. Since $w$ is increasing,
\begin{equation*}
\Ex^{\pi}\{w(X_{n+1})|X_{n}=x,I_n=i,Y_n=a\}\leq w(x+1).
\end{equation*}
From condition (c1), $w(x+1) \leq \theta w(x)$, thus 2b holds for $k=\theta$.

Finally, for 2c, iterating the above inequality $J$ times we obtain $E^{\pi}\{w(X_{n+J},I_{n+J})|X_{n}=x,I_{n}=i\}\leq w(x+J)$. Therefore it suffices to show that $\Lambda^Jw(x+J)\leq\alpha w(x)$ for some $\alpha<1$.
However the last inequality holds by condition (c2).
\end{proof}

\subsection{The Optimal Threshold Policy}\label{subsec-threspol}
In this subsection, we derive the structure of the optimal policy.
Specifically, we show in Theorem \ref{threspolicy2} that both admission and service rate controls are based on respective thresholds on the queue length. Furthermore, in Proposition \ref{optpol}
we derive a sufficient condition between the values of the economic and service
parameters, which makes the option to switch to a higher service service rate essentially of no value for
the service provider.

Let $\aad_n(x)$ be the optimal admission decision in state $(x,1)$ and
$\aser_n(x)$ the optimal service rate decision in state $(x,0)$ when
$n$ transitions remain. In addition, define
\begin{equation*}\Delta_n(x,i)=\upsilon_n(x,i)-\upsilon_n(x+1,i),\end{equation*}
as the loss in future rewards because of the increased load from an accepted arrival
and $\Delta h(x)=h(x+1)-h(x)$ the increase
in holding cost rate induced by an additional customer.

From equation \eqref{DTMDP1}, it follows
that for any $x\geq0$
\begin{equation}\label{dec-arr}
\aad_{n+1}(x)=1 \hbox{if and only if }\Delta_{n+1}(x,0)\leq R.\end{equation}

Moreover, from equation \eqref{DTMDP2}, we derive that
\begin{equation*}
\label{dec-dep}
\aser_{n+1}(x)=j\hbox{ if and only if } (\mu_j-\mu_i)\Delta_{n}(x-1,0)\geq c_j-c_i,\hbox{ for any }i\neq j.
\end{equation*}

Since $\zeta_1\leq\zeta_2\leq\ldots\leq\zeta_j$, it follows after some algebra that
\begin{equation}\label{conc_under_j} \frac{c_j-c_{j-1}}{\mu_j-\mu_{j-1}}\geq\frac{c_j-c_i}{\mu_j-\mu_i}, \hbox{ for any }i\leq j,\end{equation}
thus for $i<j$, the condition becomes $\Delta_{n}(x-1,0)\geq\zeta_j$.

Similarly, since $\zeta_{j+1}\leq\zeta_{j+2}\leq\ldots\leq\zeta_k$ we derive that
\begin{equation}\label{conc_above_j} \frac{c_{j+1}-c_j}{\mu_{j+1}-\mu_j}\leq\frac{c_i-c_j}{\mu_i-\mu_j}, \hbox{ for any }i\geq j,\end{equation}
thus for $i>j$, the condition becomes $\Delta_{n}(x-1,0)\leq\zeta_{j+1}$.

Therefore, the optimal service decision at any state $x$ can be summarized as follows
\begin{equation*}\aser_{n+1}(x)=j\mbox{ if and only if }\zeta_{j}\leq\Delta_{n}(x-1,0)\leq\zeta_{j+1}.\end{equation*}

The latter inequality implies that there exist situations where two or more service rates could be optimal. In order to make the analysis more tractable and without loss of generality, we make the following convention.
When more than one service rates are optimal in a specific state, we consider the lowest of them as the optimal. Following this convention, the optimal service rate decision can be rewritten as follows
\begin{equation}
\label{dec-dep}
\aser_{n+1}(x)=j,\hbox{ if and only if, } \zeta_j<\Delta_{n}(x-1,0)\leq\zeta_{j+1},\hbox{ for }j=2,\ldots,k-1,\;\hbox{and}\;x>0.
\end{equation}

Finally, as we have discussed in the beginning of this section,
\begin{equation}\label{dec-emp}
\aser_{n+1}(0)=1.
\end{equation}

In order to characterize the optimal policy, we first present some intermediate properties.
Lemma \ref{burdenprp1} shows that the value function is nonincreasing in $x$.

\begin{lemma}
\label{burdenprp1} The value function $\upsilon_{n}(x,i)$ is
nonincreasing in $x$.
\end{lemma}

\begin{proof}
We will prove that $\upsilon_{n}(x,i)$
is nonincreasing in $x$ for any $i=0,1$, or equivalently that $\Delta_{n}(x,i)\geq 0$, by induction on $n$.
Obviously, 
$\Delta_{0}(x,i)=0$ and
the statement holds for $n=0$.

Assume that
$\upsilon_{n}(x,i)$ is
nonincreasing in $x$ for $n$. Then, for $n+1$, we consider two
cases for $i$.

\textit{Case I: $i=0$.} First, for $x=0$, by \eqref{DTMDP2} and
\eqref{DTMDP3}, we obtain:
\begin{eqnarray*}
\Delta_{n+1}(0,0)&=& \upsilon_{n+1}(0,0)-\upsilon_{n+1}(1,0)\\
                 &=& \Delta h(0)-c_1+\lambda\Delta_{n}(0,1)-\max_{j=1,\ldots,k}\{-c_j-(\mu_k-\mu_j)\Delta_{n}(0,0)\}\\
                 &=& \Delta h(0)+\lambda\Delta_{n}(0,1)+\min_{j=1,\ldots,k}\{c_j-c_1+(\mu_k-\mu_j)\Delta_{n}(0,0)\}\geq0
\end{eqnarray*}
from the induction hypothesis and the monotonicity of $h(x)$.

For $x>0$, the terms of \eqref{DTMDP2} are nonincreasing functions of $x$ by the induction hypothesis, the assumption that $h(x)$ is increasing in $x$ and the fact that the maximum function of nonincreasing functions is
nonincreasing.

Therefore $\upsilon_{n+1}(x,0)$ is nonincreasing in $x$ for any $n$.

\textit{Case II: $i=1$.} For $i=1$, we obtain similarly that
$\upsilon_{n+1}(x,1)$ is nonincreasing in $x$ from \eqref{DTMDP1} and Case I.

Therefore the statement holds for $n+1$ and the proof is complete.
\end{proof}

The monotonicity of $\upsilon_n(x,i)$ in $x$ is intuitive. It implies
that $\Delta_n(x,i)$ is nonnegative, thus it can be seen as a burden
or profit reduction induced by one additional customer in state
$(x,i)$.

In the next theorem we show that the optimal policy is characterized
by service and admission thresholds. We first define a generic
threshold-type function that will be used in all the results.

For a function $f:{\bf{N_0}}\rightarrow{\mbox{I\!R}}$ and
$\theta\in \mbox{I\!R}$ define
\begin{equation}\label{thresgenfcn}
T_{f}(\theta)=\sup\{k\geq0:\;f(k)\leq\theta\},
\end{equation}
 with the convention $\sup\emptyset=-1$.

It is easy to see that $T_{f}(\theta)$ has the following
properties:
\begin{enumerate}
\item[(a)] $T_{f}(\theta)$ is non-decreasing in $\theta$ for all
non-decreasing functions $f$.
\item[(b)] If $f,g$ are such that $f(k)\leq g(k)$ for all
$k=0,1,\ldots$, then $T_f(\theta)\geq T_{g}(\theta)$ for any
$\theta\in R$.
\end{enumerate}

We now proceed to the Theorem.
\begin{theorem}
 \label{threspolicy2}
\begin{center}
    \begin{enumerate}
     \item[]
     \item[{\bf{i.}}] The value function $\upsilon_n(x,i)$ is
     concave in $x$, for $i=0,1$.
     \item[{\bf{ii.}}] There exist service thresholds $B^s_{n,j}$ for $j=1,\ldots,k$ and admission thresholds $B^d_{n}$ such that:
\begin{equation}\label{opdecempty2}
\aser_{n+1}(0)=1,
\end{equation}
\begin{equation}\label{thresA2}
\aser_{n+1}(x)=j,\mbox{ if and only if, }B^s_{n,j-1}<x\leq B^s_{n,j}, \mbox{ for } x>0, \mbox{ and } j=1,\ldots,k,
\end{equation}

and,
\begin{equation}\label{thresB2}
\aad_{n+1}(x)=1,\mbox{ if and only if, }x\leq B^d_{n+1}, \mbox{ for } x\geq0,
\end{equation}
where
\begin{equation*}\label{defthresA2}
B^s_{n,j}=1+T_{\Delta_n(\cdot,0)}(\zeta_{j+1}),\;j=1,\ldots,k
\end{equation*}
and,
\begin{equation*}\label{defthresB2}
B^d_{n}=T_{\Delta_n(\cdot,0)}(R).
\end{equation*}
\end{enumerate}
\end{center}
\end{theorem}

\begin{proof}
The concavity of $\upsilon_n(x,i)$ in state $(x,i)$ is equivalent to
$\Delta_n(x,i)$ being nondecreasing in $x$. The proof is by induction on $n$.
\\
For $n=0$, since $\upsilon_0(x,i)=0$, we obtain
$\Delta_0(x,i)=0$.
\\
It follows that $\aser_1(x)=1$ for all $x\geq0$, thus
\eqref{opdecempty2} and \eqref{thresA2} hold with $B^s_{0,j}=+\infty,$ for $j=1,\ldots,k$.
\\
Furthermore, from \eqref{DTMDP2} and \eqref{DTMDP3},
$\Delta_1(x,0)=\Delta h(x),\;x\geq0$, where $\Delta h(x)$ is
increasing by assumption, thus \eqref{thresB2} holds with
$B^d_1=\sup\{k\geq1, \Delta h(k)\leq R\}=T_{\Delta_1(\cdot,0)}(R)$.
\\
Now suppose that {\bf{i.}} holds for some $n$. In order to
prove the theorem, it suffices to show that {\bf{ii.}} holds for $n$
and {\bf{i.}} holds for $n+1$. To do this we will prove the
following claims in sequence:

\begin{enumerate}
\item[(a)] \eqref{thresA2} holds for $n$.
\item[(b)] $\Delta_{n+1}(x,0)$ is nondecreasing in $x$.
\item[(c)] \eqref{thresB2} holds for $n$.
\item[(d)] $\Delta_{n+1}(x,1)$ is nondecreasing in $x$.
\end{enumerate}

(a) Let $B^s_{n,j}=\sup\{x\geq1:\;\Delta_{n}(x-1,0)\leq\zeta_{j+1}\}$ for $j=0,\ldots,k$.
\\
Note that $B^s_{n,j}=1+T_{\Delta_n(\cdot,0)}(\zeta_{j+1})$, for $j=1,\ldots,k-1$ and $B^s_{n,0}=0,\;B^s_{n,k}=+\infty$.
\\
Since, by the induction hypothesis, $\Delta_n(x,0)$ is
nondecreasing in $x$ and $\zeta_1\leq\zeta_2\leq\ldots\leq\zeta_{k+1}$, it follows from \eqref{dec-dep} that
\begin{equation}\label{dec-thresA}
\aser_{n+1}(x)=j,\mbox{ if and only if, }B^s_{n,j-1}<x\leq B^s_{n,j}, \mbox{ for } x>0, \mbox{ and } j=1,\ldots,k,
\end{equation}
which completes (a).
\\
Moreover, we derive that $B^s_{n,0}\leq B^s_{n,1}\leq\ldots\leq B^s_{n,k}$.
\\
(b) Under (a), \eqref{DTMDP2} is transformed to
\begin{equation}\label{DTMDP-thresA}
\upsilon_{n+1}(x,0)=-h(x)+\lambda\upsilon_{n}(x,1)-c_j+\mu_j\upsilon_{n}(x-1,0)+(\mu_k-\mu_j)\upsilon_{n}(x,0),
\end{equation}
if and only if, $B^s_{n,j-1}<x\leq B^s_{n,j},$ for $j=1,\ldots,k$.\\

From the convexity of $h(x)$ and the concavity of $\upsilon_{n}(x,i)$ for $i=0,1$ by the induction hypothesis, we derive that $\upsilon_{n+1}(x,0)$ is concave in $x$, for all $x$ except from the threshold values. Thus, in order to complete the proof of (b), we must show that the following inequalities hold.

\begin{enumerate}
\item[$(I_1)$] $\Delta_{n+1}(0,0)\leq\Delta_{n+1}(1,0)$.
\item[$(I_2)$] $\Delta_{n+1}(B^s_{n,j}-1,0)\leq\Delta_{n+1}(B^s_{n,j},0)\leq\Delta_{n+1}(B^s_{n,j}+1,0)$ for any $j=1,\ldots,k-2.$
\item[$(I_3)$] $\Delta_{n+1}(B^s_{n,k-1}-1,0)\leq\Delta_{n+1}(B^s_{n,k-1},0)$
\end{enumerate}

$(I_1):$ We must consider the following cases for the service rate employed at states $x=1$ and $x=2$.
\\

{\it{Case 1A:}} Let $\aser_{n+1}(1)=j$ and $\aser_{n+1}(2)=j$ for some $j$. By \eqref{DTMDP3} and \eqref{DTMDP-thresA} we obtain that
\begin{eqnarray}
\nonumber  \Delta_{n+1}(0,0)&=& \Delta h(0)+c_j-c_1+\lambda\Delta_{n}(0,1)+(\mu_k-\mu_j)\Delta_{n}(0,0),\\
\nonumber  \Delta_{n+1}(1,0)&=& \Delta h(1)+\lambda\Delta_{n}(1,1)+\mu_j\Delta_{n}(0,0)+(\mu_k-\mu_j)\Delta_{n}(1,0).
\end{eqnarray}
From the convexity of $h(x)$ and the induction hypothesis, we derive that $\Delta h(0)\leq\Delta h(1)$ and $\Delta_{n}(0,i)\leq\Delta_{n}(1,i)$ for $i=0,1$.

In addition, from \eqref{conc_under_j} it follows that
\begin{equation}\label{decI1A}
\frac{c_j-c_{1}}{\mu_j-\mu_{1}}\leq\frac{c_j-c_{j-1}}{\mu_j-\mu_{j-1}}=\zeta_{j}<\Delta_{n}(0,0),\mbox{ for }j=2,\ldots,k,
\end{equation}
since $\aser_{n+1}(1)=j$.\\

Thus, from \eqref{decI1A} and Lemma \ref{burdenprp1} we derive that $c_j-c_{1}-\mu_j\Delta_{n}(0,0)<-\mu_{1}\Delta_{n}(0,0)<0$ which completes the proof of $(I_1)$ in case $1A$.
\\

{\it{Case 1B:}} Let $\aser_{n+1}(1)=j$ and $\aser_{n+1}(2)=j'>j$ for some $j$. By \eqref{DTMDP3} and \eqref{DTMDP-thresA} we obtain that
\begin{eqnarray}
\nonumber  \Delta_{n+1}(0,0)&=& \Delta h(0)+c_j-c_1+\lambda\Delta_{n}(0,1)+(\mu_k-\mu_j)\Delta_{n}(0,0),\\
\nonumber  \Delta_{n+1}(1,0)&=& \Delta h(1)+c_{j'}-c_j+\lambda\Delta_{n}(1,1)+\mu_j\Delta_{n}(0,0)+(\mu_k-\mu_{j'})\Delta_{n}(1,0).
\end{eqnarray}
From the convexity of $h(x)$ and the induction hypothesis, we derive that $\Delta h(0)\leq\Delta h(1)$ and $\Delta_{n}(0,i)\leq\Delta_{n}(1,i)$ for $i=0,1$.

In addition, since $\aser_{n+1}(1)=j$, it follows from \eqref{conc_above_j} and \eqref{dec-dep} that
\begin{equation*}\label{decI1B}\Delta_{n}(0,0)\leq\zeta_{j+1}=\frac{c_{j+1}-c_j}{\mu_{j+1}-\mu_j}\leq\frac{c_{j'}-c_j}{\mu_{j'}-\mu_j},\mbox{ for }j=2,\ldots,k.\end{equation*}
\\
From this inequality and the induction hypothesis we derive that
\begin{eqnarray}\label{decI1Bb}\nonumber (\mu_k-\mu_j)\Delta_{n}(0,0)&\leq& c_{j'}-c_j+(\mu_k-\mu_{j'})\Delta_{n}(0,0)\\\label{decI1Bb}&\leq& c_{j'}-c_j+(\mu_k-\mu_{j'})\Delta_{n}(1,0).\end{eqnarray}

Finally, since inequality \eqref{decI1A} still hold, then similarly from Lemma \ref{burdenprp1} and \eqref{decI1Bb} we prove that $(I_1)$ also holds for case $1B$.
\\

$(I_2):$ As in the proof of $(I_1)$, we must consider the following cases for the service rate employed at states $(B^s_{n,j}-1,0), (B^s_{n,j},0), (B^s_{n,j}+1,0)$ and $(B^s_{n,j}+2,0)$ for any $j=1,\ldots,k-2$.\\

{\it{Case 2A:}} Let $\aser_{n+1}(B^s_{n,j}-1)=j,\; \aser_{n+1}(B^s_{n,j})=j,\; \aser_{n+1}(B^s_{n,j}+1)=j+1$ and $\aser_{n+1}(B^s_{n,j}+2)=j+1$ for any $j=1,\ldots,k-2$.\\

By \eqref{DTMDP-thresA}, we obtain the following after some algebra for any $j=1,\ldots,k-2$:
\begin{eqnarray}
\nonumber  \Delta_{n+1}(B^s_{n,j}-1,0)&=& \upsilon_{n+1}(B^s_{n,j}-1,0)-\upsilon_{n+1}(B^s_{n,j},0)\\
\nonumber   &=&\Delta h(B^s_{n,j}-1)+\lambda\Delta_{n}(B^s_{n,j}-1,1)+\mu_j\Delta_{n}(B^s_{n,j}-2,0)\label{memberA1}\\
&&+(\mu_k-\mu_j)\Delta_{n}(B^s_{n,j}-1,0),\\
\nonumber\Delta_{n+1}(B^s_{n,j},0)&=& \upsilon_{n+1}(B^s_{n,j},0)-\upsilon_{n+1}(B^s_{n,j}+1,0)\\
\nonumber &=&c_{j+1}-c_j+\Delta h(B^s_{n,j})+\lambda\Delta_{n}(B^s_{n,j},1)+\mu_j\Delta_{n}(B^s_{n,j}-1,0)\\
&&+(\mu_k-\mu_{j+1})\Delta_{n}(B^s_{n,j},0),\label{memberA2}
\end{eqnarray}
 and,
\begin{eqnarray}
\nonumber\Delta_{n+1}(B^s_{n,j}+1,0)&=& \upsilon_{n+1}(B^s_{n,j}+1,0)-\upsilon_{n+1}(B^s_{n,j}+2,0)\\
\nonumber &=&\Delta h(B^s_{n,j}+1)+\lambda\Delta_{n}(B^s_{n,j}+1,1)+\mu_{j+1}\Delta_{n}(B^s_{n,j},0)\\
&&+(\mu_k-\mu_{j+1})\Delta_{n}(B^s_{n,j}+1,0).\label{memberA3}
\end{eqnarray}

From the convexity of $h(x)$ and the induction hypothesis,
we obtain:
\begin{equation}\label{ordI2A1}
\Delta h(B^s_{n,j}-1)\leq\Delta h(B^s_{n,j})\leq\Delta h(B^s_{n,j}+1),
\end{equation}
\begin{equation}\label{ordI2A2}
\Delta_{n}(B^s_{n,j}-1,1)\leq\Delta_{n}(B^s_{n,j},1)\leq\Delta_{n}(B^s_{n,j}+1,1),
\end{equation}
and
\begin{equation}\label{ordI2A3}
\Delta_{n}(B^s_{n,j}-2,0)\leq\Delta_{n}(B^s_{n,j}-1,0)\leq\Delta_{n}(B^s_{n,j},0)\leq\Delta_{n}(B^s_{n,j}+1,0).
\end{equation}

Since $\aser_{n+1}(B^s_{n,j})=j$ and $\aser_{n+1}(B^s_{n,j}+1)=j+1$, the following inequality hold from \eqref{dec-dep} and the definition of $B^s_{n,j}$,
\begin{equation}\label{decI2A}
\Delta_{n}(B^s_{n,j}-1,0)\leq \zeta_{j+1}=\frac{c_{j+1}-c_j}{\mu_{j+1}-\mu_j}<\Delta_{n}(B^s_{n,j},0).
\end{equation}

From \eqref{ordI2A3} and the first inequality of \eqref{decI2A}
\begin{eqnarray}\label{orddecI2Aa}
\nonumber(\mu_k-\mu_j)\Delta_{n}(B^s_{n,j}-1,0)&\leq& c_{j+1}-c_j+(\mu_k-\mu_{j+1})\Delta_{n}(B^s_{n,j}-1,0)\\&\leq& c_{j+1}-c_j+(\mu_k-\mu_{j+1})\Delta_{n}(B^s_{n,j},0).
\end{eqnarray}

Similarly, from \eqref{ordI2A3} and the second inequality of \eqref{decI2A}
\begin{equation}\label{orddecI2Ab}
c_{j+1}-c_j+\mu_j\Delta_{n}(B^s_{n,j}-1,0)\leq c_{j+1}-c_j+\mu_j\Delta_{n}(B^s_{n,j},0)<\mu_{j+1}\Delta_{n}(B^s_{n,j},0),
\end{equation}

By \eqref{memberA1} through \eqref{memberA3} and the above inequalities given in \eqref{ordI2A1} through \eqref{orddecI2Ab}, verify that $(I_2)$ is true for Case $2A$.\\

{\it{Case 2B:}} Let $\aser_{n+1}(B^s_{n,j}-1)=j,\; \aser_{n+1}(B^s_{n,j})=j,\; \aser_{n+1}(B^s_{n,j}+1)=j+1$ and $\aser_{n+1}(B^s_{n,j}+2)=j+2$ for any $j=1,\ldots,k-2$.\\

From case $2A$, we have proved that $\Delta_{n+1}(B^s_{n,j}-1,0)\leq\Delta_{n+1}(B^s_{n,j},0)$ and it remains to show that $\Delta_{n+1}(B^s_{n,j},0)\leq\Delta_{n+1}(B^s_{n,j}+1,0)$.\\

By \eqref{DTMDP-thresA}, we obtain the following after some algebra for any $j=1,\ldots,k-2$:
\begin{eqnarray}
\nonumber\Delta_{n+1}(B^s_{n,j},0)&=& \upsilon_{n+1}(B^s_{n,j},0)-\upsilon_{n+1}(B^s_{n,j}+1,0)\\
\nonumber &=&c_{j+1}-c_j+\Delta h(B^s_{n,j})+\lambda\Delta_{n}(B^s_{n,j},1)+\mu_j\Delta_{n}(B^s_{n,j}-1,0)\\
&&+(\mu_k-\mu_{j+1})\Delta_{n}(B^s_{n,j},0), \label{memberI2B2}
\end{eqnarray}
and,
\begin{eqnarray}
\nonumber\Delta_{n+1}(B^s_{n,j}+1,0)&=& \upsilon_{n+1}(B^s_{n,j}+1,0)-\upsilon_{n+1}(B^s_{n,j}+2,0)\\
\nonumber &=&c_{j+2}-c_{j+1}+\Delta h(B^s_{n,j}+1)+\lambda\Delta_{n}(B^s_{n,j}+1,1)+\mu_{j+1}\Delta_{n}(B^s_{n,j},0)\\
&&+(\mu_k-\mu_{j+2})\Delta_{n}(B^s_{n,j}+1,0).\label{memberI2B3}
\end{eqnarray}

From the convexity of $h(x)$ and the induction hypothesis,
we obtain:
\begin{equation}\label{ordI2B1}
\Delta h(B^s_{n,j})\leq\Delta h(B^s_{n,j}+1),
\end{equation}
\begin{equation}\label{ordI2B2}
\Delta_{n}(B^s_{n,j},1)\leq\Delta_{n}(B^s_{n,j}+1,1),
\end{equation}
and
\begin{equation}\label{ordI2B3}
\Delta_{n}(B^s_{n,j}-1,0)\leq\Delta_{n}(B^s_{n,j},0)\leq\Delta_{n}(B^s_{n,j}+1,0).
\end{equation}

Since $\aser_{n+1}(B^s_{n,j})=j, \aser_{n+1}(B^s_{n,j}+1)=j+1$ and $\aser_{n+1}(B^s_{n,j}+2)=j+2$, the following inequality hold from \eqref{dec-dep} and the definitions of $B^s_{n,j}$ and $B^s_{n,j+1}$,
\begin{equation}\label{decI2B}
\zeta_{j+1}=\frac{c_{j+1}-c_j}{\mu_{j+1}-\mu_j}<\Delta_{n}(B^s_{n,j},0)\leq\zeta_{j+2}=\frac{c_{j+2}-c_{j+1}}{\mu_{j+2}-\mu_{j+1}}.
\end{equation}

From \eqref{ordI2B3} and the first inequality of \eqref{decI2B}
\begin{equation}\label{orddecI2Ba}
c_{j+1}-c_j+\mu_j\Delta_{n}(B^s_{n,j}-1,0)\leq c_{j+1}-c_j+\mu_j\Delta_{n}(B^s_{n,j},0)<\mu_{j+1}\Delta_{n}(B^s_{n,j},0).
\end{equation}

Similarly, from \eqref{ordI2B3} and the second inequality of \eqref{decI2B}
\begin{eqnarray}\label{orddecI2Bb}
\nonumber(\mu_k-\mu_{j+1})\Delta_{n}(B^s_{n,j},0)&\leq& c_{j+2}-c_{j+1}+(\mu_k-\mu_{j+2})\Delta_{n}(B^s_{n,j},0)\\&\leq& c_{j+2}-c_{j+1}+(\mu_k-\mu_{j+2})\Delta_{n}(B^s_{n,j}+1,0),
\end{eqnarray}

By \eqref{memberI2B2}, \eqref{memberI2B3} and the inequalities \eqref{ordI2B1} through \eqref{orddecI2Bb}, it follows that $\Delta_{n+1}(B^s_{n,j}-1,0)\leq\Delta_{n+1}(B^s_{n,j},0)$, thus $(I_2)$ is also true for Case $2B$.\\

{\it{Case 2C:}} Let $\aser_{n+1}(B^s_{n,j}-1)=j-1,\; \aser_{n+1}(B^s_{n,j})=j,\; \aser_{n+1}(B^s_{n,j}+1)=j+1$ and $\aser_{n+1}(B^s_{n,j}+2)=j+1$ for any $j=1,\ldots,k-2$.\\

From case $2A$, we have proved that $\Delta_{n+1}(B^s_{n,j},0)\leq\Delta_{n+1}(B^s_{n,j}+1,0)$ and it remains to show that $\Delta_{n+1}(B^s_{n,j}-1,0)\leq\Delta_{n+1}(B^s_{n,j},0)$.\\

By \eqref{DTMDP-thresA}, we obtain the following after some algebra for any $j=1,\ldots,k-2$:
\begin{eqnarray}
\nonumber  \Delta_{n+1}(B^s_{n,j}-1,0)&=& \upsilon_{n+1}(B^s_{n,j}-1,0)-\upsilon_{n+1}(B^s_{n,j},0)\\
\nonumber  &=&c_{j}-c_{j-1}+\Delta h(B^s_{n,j}-1)+\lambda\Delta_{n}(B^s_{n,j}-1,1)+\mu_{j-1}\Delta_{n}(B^s_{n,j}-2,0)\label{memberC1}\\
&&+(\mu_k-\mu_{j})\Delta_{n}(B^s_{n,j}-1,0)
\end{eqnarray}
 and
\begin{eqnarray}
\nonumber\Delta_{n+1}(B^s_{n,j},0)&=& \upsilon_{n+1}(B^s_{n,j},0)-\upsilon_{n+1}(B^s_{n,j}+1,0)\\
\nonumber &=&c_{j+1}-c_j+\Delta h(B^s_{n,j})+\lambda\Delta_{n}(B^s_{n,j},1)+\mu_j\Delta_{n}(B^s_{n,j}-1,0)\\
&&+(\mu_k-\mu_{j+1})\Delta_{n}(B^s_{n,j},0).\label{memberC2}
\end{eqnarray}

From the convexity of $h(x)$ and the induction hypothesis,
we obtain:
\begin{equation}\label{ordI2C1}
\Delta h(B^s_{n,j}-1)\leq\Delta h(B^s_{n,j}),
\end{equation}
\begin{equation}\label{ordI2C2}
\Delta_{n}(B^s_{n,j}-1,1)\leq\Delta_{n}(B^s_{n,j},1),
\end{equation}
and
\begin{equation}\label{ordI2C3}
\Delta_{n}(B^s_{n,j}-2,0)\leq\Delta_{n}(B^s_{n,j}-1,0)\leq\Delta_{n}(B^s_{n,j},0).
\end{equation}

Since, $\aser_{n+1}(B^s_{n,j}-1)=j-1, \aser_{n+1}(B^s_{n,j})=j$ and $\aser_{n+1}(B^s_{n,j}+1)=j+1$ the following inequality holds from \eqref{dec-dep} and the definitions of $B^s_{n,j-1}$ and $B^s_{n,j}$,
\begin{equation}\label{decI2C}
\zeta_{j}=\frac{c_{j}-c_{j-1}}{\mu_{j}-\mu_{j-1}}<\Delta_{n}(B^s_{n,j}-1,0)\leq\zeta_{j}=\frac{c_{j+1}-c_{j}}{\mu_{j+1}-\mu_{j}}.
\end{equation}

From \eqref{ordI2C3} and the first inequality of \eqref{decI2C}
\begin{equation}\label{orddecI2Ca}
c_{j}-c_{j-1}+\mu_{j-1}\Delta_{n}(B^s_{n,j}-2,0)\leq c_{j}-c_{j-1}+\mu_{j-1}\Delta_{n}(B^s_{n,j}-1,0)<\mu_{j}\Delta_{n}(B^s_{n,j}-1,0).
\end{equation}

Similarly, from \eqref{ordI2C3} and the second inequality of \eqref{decI2C}
\begin{eqnarray}\label{orddecI2Cb}
\nonumber(\mu_k-\mu_{j})\Delta_{n}(B^s_{n,j}-1,0)&\leq &c_{j+1}-c_{j}+(\mu_k-\mu_{j+1})\Delta_{n}(B^s_{n,j}-1,0)\\&\leq& c_{j+1}-c_{j}+(\mu_k-\mu_{j+1})\Delta_{n}(B^s_{n,j},0),
\end{eqnarray}

By \eqref{memberC1}, \eqref{memberC2} and the inequalities \eqref{ordI2C1} through \eqref{orddecI2Cb}, it follows that $\Delta_{n+1}(B^s_{n,j}-1,0)\leq\Delta_{n+1}(B^s_{n,j},0)$, thus $(I_2)$ is also true for Case $2C$.\\

{\it{Case 2D:}} Let $\aser_{n+1}(B^s_{n,j}-1)=j-1,\; \aser_{n+1}(B^s_{n,j})=j,\; \aser_{n+1}(B^s_{n,j}+1)=j+1$ and $\aser_{n+1}(B^s_{n,j}+2)=j+2$ for any $j=1,\ldots,k-2$.\\

Cases $2A,\; 2B$ and $2C$, imply that inequality $(I_2)$ is also true for the service policy assumed in case $2D$ for any $j=1,\ldots,k-2$ and the proof of $(I_2)$ is complete.\\

$(I_3):$ By \eqref{DTMDP3} and \eqref{DTMDP-thresA}, we obtain that
\begin{eqnarray}
\nonumber  \Delta_{n+1}(B^s_{n,k-1}-1,0)&=& \Delta h(B^s_{n,k-1}-1)+\lambda\Delta_{n}(B^s_{n,k-1}-1,1)+\mu_{k-1}\Delta_{n}(B^s_{n,k-1}-2,0)\\
&&\nonumber+(\mu_k-\mu_{k-1})\Delta_{n}(B^s_{n,k-1}-1,0),\\
\nonumber  \Delta_{n+1}(B^s_{n,k-1},0)&=& c_k-c_{k-1}+\Delta h(B^s_{n,k-1})+\lambda\Delta_{n}(B^s_{n,k-1},1)+\mu_{k-1}\Delta_{n}(B^s_{n,k-1}-1,0)\\
&&+\mu_k\Delta_{n}(B^s_{n,k-1},0).
\end{eqnarray}

Once again, from the convexity of $h(x)$ and the induction hypothesis, we obtain that
\begin{eqnarray*}\Delta h(B^s_{n,k-1}-1)&\leq&\Delta h(B^s_{n,k-1}),\\ \Delta_{n}(B^s_{n,k-1}-1,1)&\leq&\Delta_{n}(B^s_{n,k-1},1)\end{eqnarray*} and
\begin{equation*} \Delta_{n}(B^s_{n,k-1}-2,0)\leq\Delta_{n}(B^s_{n,k-1}-1,0)\leq\Delta_{n}(B^s_{n,k-1},0).\end{equation*}
In addition, from Lemma \ref{burdenprp1}, it follows that $c_k-c_{k-1}+\mu_{k-1}\Delta_{n}(B^s_{n,k-1}-1,0)\geq0$, which completes the proof of $(I_3)$.\\

(c) Let $B^d_{n+1}=\sup\{x\geq0:\; \Delta_{n+1}(x,0)\leq
R\}$. Note that $B^d_{n}=T_{\Delta_n(\cdot,0)}(R)$.\\
 Since from (b) $\Delta_{n+1}(x,0)$ is nondecreasing in $x$,
 it follows from \eqref{dec-arr} that
\begin{equation}\label{dec-thresB}
\aad_{n+1}(x)=\left\{\begin{array}{ll}
                       1, & \hbox{if } x\leq B^d_{n+1}\\
                       0, & \hbox{if } x>B^d_{n+1}\\
                     \end{array}
                        \right.
\end{equation}

(d) Under (c), optimality equation \eqref{DTMDP1} is
transformed to
\begin{equation}\label{DTMDP-thresB}
\upsilon_{n+1}(x,1)=\left\{\begin{array}{ll}
                               R+\upsilon_{n+1}(x+1,0), & \hbox{if } x\leq B^d_{n+1}\\
                               \upsilon_{n+1}(x,0), & \hbox{if } x>B^d_{n+1}\\
                           \end{array}
                           \right.
\end{equation}
As in (b), by \eqref{DTMDP-thresB} we obtain that
$\upsilon_{n+1}(x,1)$ is concave in $x$ for $x\leq B^d_{n+1}-2$ and
for $x\geq B^d_{n+1}+1$,
 because $\upsilon_{n+1}(x,0)$ is
concave in $x$, as we have proved in (b).\\
In order to complete the proof we have to consider the cases $x=B^d_{n+1}-1$ and $x=B^d_{n+1}$, thus we need to show
\begin{equation}\label{ineqadmthres}
\Delta_{n+1}(B^d_{n+1}-1,1)\leq\Delta_{n+1}(B^d_{n+1},1)\leq\Delta_{n+1}(B^d_{n+1}+1,1).
\end{equation}
 After some algebra we obtain that
\begin{eqnarray}
\Delta_{n+1}(B^d_{n+1}-1,1)&=&\upsilon_{n+1}(B^d_{n+1}-1,1)-\upsilon_{n+1}(B^d_{n+1},1) \nonumber \\
    & = & \Delta_{n+1}(B^d_{n+1},0),\label{memberB1}\\
\Delta_{n+1}(B^d_{n+1},1)&=&\upsilon_{n+1}(B^d_{n+1},1)-\upsilon_{n+1}(B^d_{n+1}+1,1)\nonumber \\
    & = & R,\label{memberB2}\\
\Delta_{n+1}(B^d_{n+1}+1,1)&=&\upsilon_{n+1}(B^d_{n+1}+1,1)-\upsilon_{n+1}(B^d_{n+1}+2,1) \nonumber \\
    &=& \Delta_{n+1}(B^d_{n+1}+1,0).\label{memberB3}
\end{eqnarray}
 By the definition of $B^d_{n+1}$, it follows that
\begin{equation}\label{ordB1}
\Delta_{n+1}(B^d_{n+1},0)\leq R<\Delta_{n+1}(B^d_{n+1}+1,0),
\end{equation}
 By inequalities \eqref{memberB1}-\eqref{ordB1},
\eqref{ineqadmthres} holds and this completes the proof of the theorem.
\end{proof}

According to Theorem \ref{threspolicy2}, the optimal action in state
$(x,1)$ with $n$ remaining transitions is twofold and prescribed by the pair
$(\aad_n(x), \aser_n(x))$. It is optimal to accept the incoming
customer if $x \leq B^d_n$ and reject him otherwise, whereas the optimal
service rate for the time interval until the next transition is
determined immediately after the admission action is taken, and the
service rate is determined by an increasing sequence of service thresholds $B^s_{n-1,0}\leq B^s_{n-1,1}\leq\ldots\leq B^s_{n-1,k}$ depending on the value of $x$. The difference in the subscript between the admission and
service thresholds is due to the fact that for the admission decision with $n$ remaining steps the relevant burden function is $\Delta_n(\cdot,0)$, while for the service rate decision it is $\Delta_{n-1}(\cdot,0)$.
On the other hand, in states $(x,0)$ the single optimal action is determined solely by
$\aser_n(x)$, the service rate to be employed until the next transition epoch.

In the remainder of the paper it will be useful to adopt an alternative viewpoint and
consider the optimal policy of admission/service control as pairs of
decisions both taken at departure epochs. Specifically let
\begin{equation*}
\act_{n+1}(x)=\left ( \aser_{n+1}(x),\aad_n(x) \right ).
\end{equation*}
The pair $\act_{n+1}(x)$ can be seen as a decision made at state $(x,0)$ with $n+1$ remaining steps,
prescribing: (i) the service rate to be
employed until the next transition epoch and (ii) whether to admit a new customer in the event that
the next transition is an arrival.
Thus, one may view the admission/rejection policy
as a sign posted at the entrance of the system after every departure
event. The sign specifies whether new arrivals are welcome to enter
the system or not. According to this interpretation, the pair $\act_{n+1}(x)$ specifies which
service rate will be employed when $n+1$ transitions remain, as well
as which sign will be posted at that instant.

An interesting question which is related to the issue of service flexibility discussed next, is which of the available service rates are indeed useful. The following proposition establishes the relative ordering between the admission and service thresholds according to the
values of economic parameters $R,\;c_j$ and service parameters $\mu_j$ for $j=1,\ldots,k$, which in part answers this question. More specifically, we let $j_0=\min\{j\geq1: R\leq\zeta_{j+1}\}$ and we show that the availability of service rates higher than $\mu_{j_0}$
is of limited value as a profit maximizing option. Indeed, whenever a service rate higher than $\mu_{j_0}$ is employed in a state $x$, the
rejection sign is posted for arrivals at the next decision epoch.

\begin{proposition}\label{optpol}
The following hold.
\begin{center}
    \begin{enumerate}
     \item[]
     \item[{\bf{i.}}] $B^d_{n}+1\leq B^s_{n,j}$ for any $j>j_0.$
     \item[{\bf{ii.}}]
\begin{equation*}
\act_{n+1}(x)=\left\{\begin{array}{ll}
                       (j,1),& 0\leq x\leq B^d_{n}\mbox{ and } B^s_{n,j-1}<x<B^s_{n,j},\mbox{ for } j=1,\ldots,j_0\\
                       (j,0),& x>B^d_{n}\mbox{ and } B^s_{n,j-1}<x<B^s_{n,j},\mbox{ for } j=j_0,\ldots,k
                      \end{array}
                      \right.
\end{equation*}
\end{enumerate}
\end{center}
\end{proposition}

\begin{proof}
Since $\zeta_1\leq\zeta_2\leq\ldots\leq\zeta_{k+1}$, we obtain that $\zeta_{j_0+1}<\zeta_{j+1}$ for any $j>j_0$.

Thus, $R\leq\zeta_{j_0+1}<\zeta_{j+1}$ for any $j>j_0$, and, it follows that $T_{\Delta_n(\cdot,0)}(R)\leq T_{\Delta_n(\cdot,0)}(\zeta_{j+1})$ for any $j>j_0$, since $\Delta_n(x,0)$ is non decreasing in $x$.

Therefore,
\begin{equation*}
1+B^d_n=1+T_{\Delta_n(\cdot,0)}(R)\leq 1+T_{\Delta_n(\cdot,0)}(\zeta_{j+1})=B^s_{n,j},
\end{equation*}
for any $j>j_0$.

The possible cases for $\act_{n+1}(x)$ follow immediately
given the inequality $1+B^d_n\leq B^s_{n,j}$ for any $j>j_0$.
\end{proof}

When service rate $\mu_j$ is optimal, then the sign posted is admit if $j<j_0$, and reject if $j>j_0$, while in the border case $j=j_0$ the decision depend on the specific value of $x$. Based on this property, we refer to the rates $\mu_{j_0+1},\ldots,\mu_k$ as useless, since if the initial state is below $B^s_{n,j_0}$ then they will never be used, otherwise some of them will be used only for a finite number of transitions until $x<B^s_{n,j_0}$.

It follows from Proposition \ref{optpol} that when $R\leq\zeta_{j+1}$ for some $j$, then service rate $\mu_j$ is useless. However, if $R>\zeta_{j+1}$, we cannot generally make the opposite statement, i.e., that this service rate brings beneficial flexibility. Indeed, in this case it can be shown similarly that $B^d_n+1\geq B^s_{n,j}$, which includes the possibility of $B^d_n+1=B^s_{n,j}$ and falls in the case of Proposition \ref{optpol}(i). Therefore, that $R > \zeta_{j+1}$, i.e., that the cost of $\mu_{j+1}$ relative to $\mu_j$ is sufficiently low compared to the service revenue, is a necessary but not sufficient condition for this rate to be useful.

This leads to the question of the value of the service rate switch option in general.
In the following section, we explore more thoroughly this issue, by analyzing the value of service flexibility as a function of the system state.

\section{The Value of Service Flexibility}
\label{sec-VoSF}
Motivated by the discussion at the end of the previous section, we next explore how the service rate switch option affects, in terms of profit and admission thresholds, the system where only admission control is employed at the lowest rate.
In terms of profit, we show that the option to increase service capacity becomes more profitable as the system congestion increases, whereas in terms of admission thresholds, the service flexibility ensures that more customers could be accepted.

To assess the value of service rate flexibility, we note that the
restriction of the combined problem described in
\eqref{DTMDP1}-\eqref{DTMDP-int}, to the class of policies where the
service rate is always set to lowest mode, i.e. $\mu_1$, is equivalent to a pure
admission control subproblem (for a typical admission control
formulation see \cite{puterman}, p.p.568-571).  This restriction
corresponds to the following set of optimality equations in finite horizon:
\begin{eqnarray}
  \hat{\upsilon}_{n+1}(x,1)&=&\max\{R+\hat{\upsilon}_{n+1}(x+1,0),\hat{\upsilon}_{n+1}(x,0)\},\; x\geq0\label{SerFlexEq1} \\
  \hat{\upsilon}_{n+1}(x,0)&=&-h(x)-c_1+\lambda\hat{\upsilon}_{n}(x,1)+\mu_1\hat{\upsilon}_{n}(x-1,0)
  +(\mu_k-\mu_1)\hat{\upsilon}_{n}(x,0),\; x>0\label{SerFlexEq2} \\
  \hat{\upsilon}_{n+1}(0,0)&=&-h(0)-c_1+\lambda\hat{\upsilon}_{n}(0,1)+\mu_k\hat{\upsilon}_{n}(0,0)\label{SerFlexEq3}\\
  \hat{\upsilon}_{0}(x,i)&=&0,\;\;x\geq0, i\in\{0,1\},\label{SerFlexEqint}
\end{eqnarray}
 where $\hat{\upsilon}_{n}(x,i)$ denotes the maximum
discounted net profit for the remaining $n$ transitions, when the
service rate is set to $\mu_1$ and admission is dynamically
controlled. In the following we will refer to the restricted problem
as the admission control subproblem.

Let $\aadh_{n}(x)$ be the optimal decision in state $(x,1)$.
From $\eqref{SerFlexEq1}$ to $\eqref{SerFlexEq3}$ it follows that
\begin{equation}\label{dec-AC}
 \aadh_{n+1}(x)=\left\{\begin{array}{ll}
                              1, & \hbox{if } \hat{\Delta}_{n+1}(x,0)\leq R\\
                              0, & \hbox{if } \hat{\Delta}_{n+1}(x,0)> R\\
                           \end{array}
                           \right.
\end{equation}
 for $x\geq0$, where
\begin{equation*}
\hat{\Delta}_{n}(x,i)=\hat{\upsilon}_{n}(x,i)-\hat{\upsilon}_{n}(x+1,i)
\end{equation*}
 denotes the burden in terms of expected profit reduction
that an additional customer brings to the defined
system.

Similarly to Lemma \ref{burdenprp1} and Theorem \ref{threspolicy2}, it can be shown that $\hat{\upsilon}_n(x,i)$ is
non increasing and concave in $x$, or equivalently that $\hat{\Delta}_n(x,i)$ is nonnegative and nondecreasing in $x$.

Therefore, the optimal admission rule is characterized by
admission thresholds
\begin{equation*}
\hat{B}^d_n=T_{\hat{\Delta}_n(\cdot,0)}(R)
\end{equation*}
so that $\aadh_{n+1}(x)=1$ if and only if
$x\leq\hat{B}^d_{n+1},\;x\geq0,\;n=0,1,\ldots$.  The threshold
structure of the optimal policy is not new (see e.g. \cite{walrand}, p.278,
\cite{puterman}, p.568). We restate it here in a notation that allows
comparison with the combined problem.

Now, we can define the value of service flexibility as the benefit that the
system administrator obtains from using the service rate switch option,
i.e.,
$\hat{\epsilon}_{n}(x,i)=\upsilon_{n}(x,i)-\hat{\upsilon}_{n}(x,i)$.
It is immediate that $\hat{\epsilon}_{n}(x,i)\geq0$, for all $x,i,n$.

In the next theorem, we first prove that $\hat{\epsilon}_{n}(x,i)$ is nondecreasing in the system length $x$, thus the
option to switch to a higher service rate is more useful as the queue
becomes longer, which is intuitively expected.  Moreover, this is
equivalent to the fact that the burden that an additional customer
imposes on the system is lower when the option to switch to a higher service rate is
available, compared to the pure admission control subproblem. This is
also intuitive, since by increasing the service rate, it is possible
to alleviate the extra delay because of the additional customer.

Secondly, we show that the admission thresholds are increased when the service rate switch is available, thus a customer who would not
be accepted in the restricted system may be accepted when the system manager has the flexibility to switch to a higher service rate.

\begin{theorem}\label{SerFlexincr}
\begin{enumerate}
\item[]
\item[{\bf{i.}}] $\hat{\epsilon}_{n}(x,i)$, is nondecreasing in $x$ for all $n, i$.
\item[{\bf{ii.}}] $\hat{B}^d_n\leq B^d_n$ for all $n$.
\end{enumerate}
\end{theorem}

\begin{proof}
Note that {\bf{i.}} is equivalent to
$\Delta_n(x,i)\leq\hat{\Delta}_n(x,i)$ for any
$i\in\{0,1\}$ and $n=0,1,\ldots$.

The proof is by induction on $n$. For $n=0$, {\bf{i.}} is immediate since
$\hat{\epsilon}_0(x,i)=0$, by initial conditions \eqref{DTMDP-int}
and \eqref{SerFlexEqint}.

Now suppose that {\bf{i.}} holds for some $n$. We consider
the following cases:

Case 1: $i=0$.
For $x>0$ by equations \eqref{DTMDP2} and
\eqref{SerFlexEq2}, we obtain
\begin{eqnarray*}
\hat{\epsilon}_{n+1}(x,0) &=& \upsilon_{n+1}(x,0)-\hat{\upsilon}_{n+1}(x,0) \\
&=& -h(x)+\lambda\upsilon_{n}(x,1)+\max_{j=1,\ldots,k}\{-c_j+\mu_j\upsilon_{n}(x-1,0)+(\mu_k-\mu_j)\upsilon_{n}(x,0)\}\\
&&+h(x)+c_1-\lambda\hat{\upsilon}_{n}(x,1)-\mu_1\hat{\upsilon}_{n}(x-1,0)-(\mu_k-\mu_1)\hat{\upsilon}_{n}(x,0)\\
&=& \lambda\hat{\epsilon}_{n}(x,1)+\mu_1\hat{\epsilon}_{n}(x-1,0)+(\mu_k-\mu_1)\hat{\epsilon}_{n}(x,0)\\
&&+\max_{j=2,\ldots,k}\left\{\left\{c_1-c_j+(\mu_j-\mu_1)\Delta_{n}(x-1,0)\right\}^{+}\right\},
\end{eqnarray*}
where $\{\}^{+}$ refers to the positive part of the corresponding quantity.

Since $\Delta_n(x-1,0)$ is nondecreasing in $x$, it follows that the maximum function of non decreasing functions is non decreasing in $x$ and the monotonicity of  $\hat{\epsilon}_{n+1}(x,0)$ for $x>0$ follows readily by the induction hypothesis.

Finally, for $x=0$, we derive the following after some algebra
\begin{equation*}
\hat{\epsilon}_{n+1}(0,0)=\upsilon_{n+1}(0,0)-\hat{\upsilon}_{n+1}(0,0)=
\lambda\hat{\epsilon}_{n}(0,1)+\mu_k\hat{\epsilon}_{n}(0,0)
\end{equation*}
\begin{eqnarray*}
\hat{\epsilon}_{n+1}(1,0) &=& \upsilon_{n+1}(1,0)-\hat{\upsilon}_{n+1}(1,0) \\
&=& \lambda\hat{\epsilon}_{n}(1,1)+\mu_1\hat{\epsilon}_{n}(0,0)+(\mu_k-\mu_1)\hat{\epsilon}_{n}(1,0)\\
&&+\max_{j=2,\ldots,k}\left\{\left\{c_1-c_j+(\mu_j-\mu_1)\Delta_{n}(0,0)\right\}^{+}\right\}
\end{eqnarray*}
From the above equations it follows that $\hat{\epsilon}_{n+1}(1,0)-\hat{\epsilon}_{n+1}(0,0)\geq0$, by the induction hypothesis.

Therefore, $\hat{\epsilon}_{n+1}(x,0)$ is nondecreasing in
$x$, for $x\geq0$, thus
$\Delta_{n+1}(x,0)\leq\hat{\Delta}_{n+1}(x,0)$. By property (b) of
the generic threshold-type function, $T_f(\theta)$, we obtain that
\begin{equation*}
T_{\hat{\Delta}_{n+1}(\cdot,0)}(R)\leq
T_{\Delta_{n+1}(\cdot,0)}(R)\Rightarrow\hat{B}^d_{n+1}\leq B^d_{n+1},
\end{equation*}
thus {\bf{ii.}} holds for $n+1$.

Case 2: $i=1$.
From the optimality equations $\eqref{DTMDP1},\;
\eqref{SerFlexEq1}$ and the property $\hat{B}^d_{n+1}\leq B^d_{n+1}$,
$\hat{\epsilon}_{n+1}(x,1)$ can be written as
\begin{eqnarray}
\hat{\epsilon}_{n+1}(x,1)&=& \upsilon_{n+1}(x,1)-\hat{\upsilon}_{n+1}(x,1) \nonumber \\
&=& \left\{\begin{array}{ll}
             \hat{\epsilon}_{n+1}(x+1,0),& x\leq\hat{B}^d_{n+1}\\
             R+\upsilon_{n+1}(x+1,0)-\hat{\upsilon}_{n+1}(x,0),&\hat{B}^d_{n+1}<x\leq B^d_{n+1}\\
            \hat{\epsilon}_{n+1}(x,0),&x\geq B^d_{n+1}+1\\
        \end{array}
        \right. \label{SerFlexind1}
\end{eqnarray}
To show monotonicity we consider further subcases for $x$.

 {\textit{Case 2a:}} $x\leq\hat{B}^d_{n+1}$.
 For $x\leq\hat{B}^d_{n+1}-1$ by \eqref{SerFlexind1} we
obtain that $\hat{\epsilon}_{n+1}(x,1)=\hat{\epsilon}_{n+1}(x+1,0)$,
which is nondecreasing in $x$, from Case 1.

 For $x=\hat{B}^d_{n+1}$ we have to prove that
$\hat{\epsilon}_{n+1}(\hat{B}^d_{n+1},1)\leq\hat{\epsilon}_{n+1}(\hat{B}^d_{n+1}+1,1)$.

 Indeed,
\begin{eqnarray*}
\lefteqn{\hat{\epsilon}_{n+1}(\hat{B}^d_{n+1},1)-\hat{\epsilon}_{n+1}(\hat{B}^d_{n+1}+1,1)=}\\
& &\upsilon_{n+1}(\hat{B}^d_{n+1}+1,0)-\hat{\upsilon}_{n+1}(\hat{B}^d_{n+1}+1,0) \\
& &-R-\upsilon_{n+1}(\hat{B}^d_{n+1}+2,0)+\hat{\upsilon}_{n+1}(\hat{B}^d_{n+1}+1,0)\\
&=&\Delta_{n+1}(\hat{B}^d_{n+1}+1,0)-R\leq0,
\end{eqnarray*}
 by definition of $\hat{B}^d_{n+1}$.\\
 {\textit{Case 2b:}} $\hat{B}^d_{n+1}<x\leq B^d_{n+1}$.\\
 For $\hat{B}^d_{n+1}+1<x\leq B^d_{n+1}-1$, it follows from
\eqref{SerFlexind1} that,
\begin{eqnarray*}
\lefteqn{\hat{\epsilon}_{n+1}(x,1)-\hat{\epsilon}_{n+1}(x+1,1)=} \\
&&R+\upsilon_{n+1}(x+1,0)-\hat{\upsilon}_{n+1}(x,0)\\
&&-R-\upsilon_{n+1}(x+2,0)+\hat{\upsilon}_{n+1}(x+1,0)\\
&=&\Delta_{n+1}(x+1,0)-\hat{\Delta}_{n+1}(x,0)\leq0,
\end{eqnarray*}
 from Case 1.

 For $x=B^d_{n+1}$ we must show that $\hat{\epsilon}_{n+1}(B^d_{n+1},1)\leq\hat{\epsilon}_{n+1}(B^d_{n+1}+1,1)$.

 Again from \eqref{SerFlexind1} we obtain
\begin{eqnarray*}
\lefteqn{\hat{\epsilon}_{n+1}(B^d_{n+1},1)-\hat{\epsilon}_{n+1}(B^d_{n+1}+1,1) =} \\
&& R+\upsilon_{n+1}(B^d_{n+1}+1,0)-\hat{\upsilon}_{n+1}(B^d_{n+1},0)\\
&&-\upsilon_{n+1}(B^d_{n+1}+1,0)+\hat{\upsilon}_{n+1}(B^d_{n+1}+1,0)\\
&=&R-\hat{\Delta}_{n+1}(B^d_{n+1},0)\leq0,
\end{eqnarray*}
 by definition of $B^d_{n+1}$.\\
 {\textit{Case 2c:}} $x\geq B^d_{n+1}+1$.

 For $x\geq B^d_{n+1}+1$ the monotonicity of
$\hat{\epsilon}_{n+1}(x,1)$ it is immediate by
\eqref{SerFlexind1}.

 Thus, we have shown the monotonicity of
$\hat{\epsilon}_{n+1}(x,1)$ in $x$ for the case
$\hat{B}^d_{n+1}<B^d_{n+1}$.
 It remains to examine the case $\hat{B}^d_{n+1}=B^d_{n+1}$.
Then the middle range in \eqref{SerFlexind1} disappears and it is
left to show that, for
$x=B^d_{n+1}:\;\hat{\epsilon}_{n+1}(B^d_{n+1}+1,1)\geq\hat{\epsilon}_{n+1}(B^d_{n+1},1)$.

 Once again by \eqref{SerFlexind1} we obtain
\begin{equation*}
\hat{\epsilon}_{n+1}(B^d_{n+1}+1,1)-\hat{\epsilon}_{n+1}(B^d_{n+1},1)=\hat{\epsilon}_{n+1}(B^d_{n+1}+1,0)-\hat{\epsilon}_{n+1}(B^d_{n+1}+1,0)=0.
\end{equation*}
 Therefore $\hat{\epsilon}_{n+1}(x,1)$ is nondecreasing in $x$.
\end{proof}

Now that the properties of $\hat{\epsilon}_n(x,i)$ have been shown for
the finite horizon version of the problem, it is natural to ask how
the results of Proposition \ref{optpol} are related to Theorem
\ref{SerFlexincr}. In particular, one might conjecture that if $j_0=1$, then $\hat{\epsilon}_n(x,0)=0$ and
$B^d_n=\hat{B}^d_n$, since the service rate of the admission control sub-problem is set at the lowest rate $\mu_1$. However, this may not be generally true for the following reason. If $j_0=1$, then service rates higher than $\mu_1$ are not used in states where customers are admitted. However, they
may still be used in initial states with large $x$ although new arrivals are
rejected, in order to empty the queue faster and reduce the holding
costs. Therefore $\hat{\epsilon}_n(x,0)$ could still be positive in
such states. Furthermore, even for states with $x$ small enough so
that the optimal service rate is $\mu_1$ in the combined problem, i.e.,
for $x\leq B_n^d$ it is not obvious that $\hat{\epsilon}_n(x,0)=0$. If
the thresholds could be shown to be monotone with respect to the
number of periods $n$, then the above could be shown by induction,
however this monotonicity may not be true in general.

On the other hand, it has been shown in Theorem
\ref{VerificationTheorem} that, under fairly general conditions on the
holding cost function $h(x)$, the finite horizon problems converge as
$n \to \infty$, to the infinite horizon problem, for which the optimal
policy is stationary. For this limiting problem it is possible to
prove an interesting relationship between Proposition \ref{optpol} and
Theorem \ref{SerFlexincr}, as we do next.

\subsection{Service Rate Flexibility Under Infinite Horizon}\label{subsec-Asym_Res}
In this subsection we show that if $j_0=1$, i.e. $R\leq\zeta_2$, the value
of service flexibility is essentially of no value in low congestion
states and the optimal admission thresholds are the same with and
without the service rate switch option in the framework of the
infinite horizon discounted problem.

Consider the infinite horizon problem and assume that the holding cost
function satisfies the conditions of Theorem
\ref{VerificationTheorem}. It follows that the value functions of the
combined and the admission control subproblems converge to their
infinite horizon counterparts $\upsilon(\cdot, \cdot),
\hat{\upsilon}(\cdot, \cdot)$, which retain the monotonicity and
concavity properties proved for finite $n$. Thus, the infinite horizon
optimal policies are still threshold-based with time stationary
thresholds, i.e., there exist $B^s_j$ for $j=1,\ldots,k$, and $B^d, \hat{B}^d$ such that the
optimal policy $a=a(B^s_1, B^s_2,\ldots, B^s_k, B^d)$ for the combined problem is
$\aser(x,i) = j$, if and only if, $B^s_{j-1}<x\leq B^s_j$ for $j=1,\ldots,k$ with $B^s_0=0$ and $B^s_k=+\infty$, and
\ \
$\aad(x,1) =1$, if and only if, $x\leq B^d$, and the optimal policy $\hat{a}=\hat{a}(\hat{B}^d)$ for the admission control subproblem is
$\aadh(x,1) = 1$, if and only if, $x\leq \hat{B}^d$.


Furthermore, the service rate flexibility $\hat{\epsilon}(x,i) =
\upsilon(x, i) - \hat{\upsilon}(x, i)$ is nondecreasing in $x$ and
$\hat{B}^d \leq B^d$.

We thus restrict attention to the class of stationary threshold-type
policies.  Let $\pi = \pi(b^s_1,b^s_2,\ldots,b^s_k,b^d)$ be any (not necessarily optimal)
stationary threshold policy for the combined problem, prescribing
actions $\piser(x,i) = j$ if and only if $b^s_{j-1}<x\leq b^s_j$ for $j=1,\ldots,k$, and $\piad(x,i)
= 1$ if and only if $x\leq b^d$.  A threshold policy $\hat{\pi}=
\hat{\pi}(b^d)$ for the admission control subproblem can be defined
similarly.  Finally let $\upsilon_{\pi}$ denote the infinite horizon
discounted profit function for the combined problem under threshold
policy $\pi$, and $\hat{\upsilon}_{\hat{\pi}}$ the corresponding function
for the admission control subproblem under threshold policy
$\hat{\pi}$.

Now consider a threshold policy $\pi(b^s_1,b^s_2,\ldots,b^s_k,b^d)$ with thresholds $b^s_1
\geq b^d+1$ and the corresponding policy $\hat{\pi}(b^d)$ for the
admission control subproblem with the same admission threshold. The
following lemma shows that under these two policies the value
functions of the two control problems coincide for all reachable
states $x$ where the service rate $\mu_1$ is used.

\begin{lemma}\label{thresvalfnc} For any $\pi=\pi(b^s_1,b^s_2,\ldots,b^s_k,b^d),\;\hat{\pi}=\hat{\pi}(b^d)$ such that: $b^s_1\geq b^d+1$, then
\begin{equation*}
\upsilon_{\pi}(x,i)=\hat{\upsilon}_{\hat{\pi}}(x,i),\;\;x=0,1,\ldots,b^s_1,\;i=0,1.
\end{equation*}
\end{lemma}

\begin{proof}
For the infinite horizon discounted profit maximization problem the value function corresponding to stationary policy $\pi$ can be found as the unique solution to a system of linear equations corresponding to the policy evaluation step of the policy iteration method. Specifically, for policy $\pi$ with $b^s_1 > b^d$, the policy evaluation equations are
\begin{eqnarray*}
  \upsilon_{\pi}(0,0)&=&-h(0)-c_1+\lambda\upsilon_{\pi}(0,1)+\mu_k\upsilon_{\pi}(0,0) \\ 
  \upsilon_{\pi}(x,0)&=&-h(x)-c_1+\lambda\upsilon_{\pi}(x,1)+\mu_1\upsilon_{\pi}(x-1,0)+(\mu_k-\mu_1)\upsilon_{\pi}(x,0),\;x=1,\ldots,b^s_1   \\ 
  \upsilon_{\pi}(x,1)&=&R+\upsilon_{\pi}(x+1,0),\;x=0,\ldots,b^d \\ 
  \upsilon_{\pi}(x,1)&=&\upsilon_{\pi}(x,0),\;x=b^d+1,\ldots,b^s_1\\ 
  \upsilon_{\pi}(x,0)&=&-h(x) -c_j +\lambda \upsilon_{\pi}(x,1) + \mu_j \upsilon_{\pi}(x-1,0)+(\mu_k-\mu_j)\upsilon_{\pi}(x,0), b^s_{j-1}< x \leq b^s_j,\\&& \mbox{ for }j=2,\ldots,k\\ 
  \upsilon_{\pi}(x,1)&=&\upsilon_{\pi}(x,0),\;x > b^s_1 
\end{eqnarray*}
In the above system, the $2(b^s_1+1)$ values $\upsilon_{\pi}(x,i), x=0,\ldots,b^s_1, i=0,1$  are the unique solution to the first $2(b^s_1+1)$ equations for $i=0,1$ and $x=0,\ldots,b^s_1$. It is also easy to see that each of the remaining quantities $\upsilon_{\pi}(x,i)$ for $i=0,1$ and $x>b^s_1$ can be obtained recursively as a function of $\upsilon_{\pi}(b^s_1,0)$, from the remaining equations.

Similarly, for policy $\hat{\pi}(b^d)$ the policy evaluation equations of the admission control subproblem are
\begin{eqnarray*}
  \hat{\upsilon}_{\hat{\pi}}(0,0)&=&-h(0)-c_1+\lambda\hat{\upsilon}_{\hat{\pi}}(0,1)+\mu_k\hat{\upsilon}_{\hat{\pi}}(0,0)\label{infSerFlexEq3}\\
  \hat{\upsilon}_{\hat{\pi}}(x,0)&=&-h(x)-c_1+\lambda\hat{\upsilon}_{\hat{\pi}}(x,1)+\mu_1\hat{\upsilon}_{\hat{\pi}}(x-1,0)+(\mu_k-\mu_1)\hat{\upsilon}_{\hat{\pi}}(x,0),\; x=1,\ldots,b^s_1\label{infSerFlexEq2} \\
  \hat{\upsilon}_{\hat{\pi}}(x,1)&=&R+\hat{\upsilon}_{\hat{\pi}}(x+1,0),\;x=0,\ldots,b^d\label{infSerFlexEq1B}\\
  \hat{\upsilon}_{\hat{\pi}}(x,1)&=&\hat{\upsilon}_{\hat{\pi}}(x,0),\;x=b^d+1,\ldots,b^s_1\label{infSerFlexEq1A}\\
  \hat{\upsilon}_{\hat{\pi}}(x,0)&=&-h(x)+ \lambda \hat{\upsilon}_{\hat{\pi}}(x,1) + \mu_l \hat{\upsilon}_{\hat{\pi}}(x-1,0) + \delta \hat{\upsilon}_{\hat{\pi}}(x,0), \; x > b^s_1\\ 
  \hat{\upsilon}_{\hat{\pi}}(x,1)&=&\hat{\upsilon}_{\hat{\pi}}(x,0),\;x > b^s_1 
\end{eqnarray*}
and the $2(b^s_1+1)$ values $\hat{\upsilon}_{\hat{\pi}}(x,i), x=0,\ldots,b^s_1, i=0,1$  are the unique solution to the first $2(b^s_1+1)$ equations for $i=0,1$ and $x=0,\ldots,b^s_1$.

We finally note that the first $2(b^s_1+1)$ equations are identical in the two problems above, therefore
$\upsilon_{\pi}(x,i)=\hat{\upsilon}_{\hat{\pi}}(x,i)$, for $x=0,1,\ldots,b^s_1,\;i=0,1$.
\end{proof}

In the next proposition we make use of Lemma \ref{thresvalfnc} to show that if the optimal policy\\
$a(B^s_1,\ldots,B^s_k, B^d)$ for the combined problem is such that $B^s_1 \geq B^d+1$, then the service rate flexibility is equal to zero for states with $x\leq B^s_1$. Furthermore, the optimal admission threshold for the admission control subproblem is equal to that for the combined problem.

\begin{proposition}\label{cntcombadm}
If $B^s_1\geq B^d+1$, then $\hat{B}^{d}=B^d$, and
$\hat{\epsilon}(x,i)=0,\;i=0,1,\;x=0,1,\ldots,B^s_1.$
\end{proposition}

\begin{proof}
We have shown that the optimal admission thresholds generally satisfy $\hat{B}^{d}\leq B^d$.
Assume that the optimal policy for the combined problem $a(B^s_1,\ldots,B^s_k, B^d)$ satisfies $B^s_1\geq B^d+1$.

Consider the policy $\hat{\pi}=\hat{\pi}(B^d)$ for the admission control subproblem that applies admission threshold $B^d$. From Lemma \ref{thresvalfnc} it follows that
$\upsilon_a(x,i) = \hat{\upsilon}_{\hat{\pi}}(x,i)$, for $x=0,\ldots, B^s_1, i=0,1$.
However, $\upsilon(x,i)= \upsilon_a(x,i),\ \hat{\upsilon}_{\hat{\pi}}(x,i)\leq \hat{\upsilon}(x,i) \leq \upsilon(x,i)$ for all $(x,i)$. It follows that $\upsilon(x,i) = \hat{\upsilon}(x,i)$, thus $\hat{\epsilon}(x,i)=0$, for $x=0,\ldots, B^s_1, i=0,1$.

To show that $\hat{B}^d=B^d$, suppose that the optimal policy $\hat{a}(\hat{B}^d)$ for the admission control subproblem is such that $\hat{B}^{d}<  B^d$ and consider state $(\hat{B}^d+1, 1)$. By the definition of $\hat{B}^d$ it follows that admitting a customer in this state is strictly suboptimal for the admission control subproblem, i.e.,
$$
R+\hat{\upsilon}(\hat{B}^d+2, 0) < \hat{\upsilon}(\hat{B}^d+1, 0).
$$
On the other hand, for the combined problem admitting the customer in this state is optimal, i.e.,
$$
R+\upsilon(\hat{B}^d+2, 0) \geq \upsilon(\hat{B}^d+1, 0).
$$
However, $\hat{B}^d+2 \leq B^d +1 \leq B^s$, thus, as we have shown above, $\upsilon(\hat{B}^d+2, 0) = \hat{\upsilon}(\hat{B}^d+2, 0)$ and $\hat{\upsilon}(\hat{B}^d+1, 0) = \upsilon(\hat{B}^d+1, 0)$. Therefore the two inequalities above lead to a contradiction and we conclude that $\hat{B}^d = B^d$.
\end{proof}

We can now show that, for the infinite horizon case, Proposition \ref{optpol} complements Theorem \ref{SerFlexincr}, in the sense that $R \leq \zeta_2$ actually implies that adding the service rate switch possibility does not affect the admission threshold, and the value of flexibility is equal to zero for states with low congestion.
This result is an immediate consequence of Propositions \ref{optpol} and \ref{cntcombadm}.

\begin{theorem}\label{vof_novalue}
If $R\leq\zeta_2$, then $\hat{B}^{d}=B^d$, and
$\hat{\epsilon}(x,i)=0,\;i=0,1,\;x=0,1,\ldots,B^s_1.$
\end{theorem}

\section{The Average Reward Case}
\label{sec-Average}
In this section, we consider the objective of expected average reward
per unit time, and provide a sufficient condition under which a
long-run average reward optimal policy exists and is obtained as a limit of the discounted reward problems as the discount rate $\beta \downarrow 0$. This implies that the results on the structure of the optimal policy and the value
of service rate flexibility presented in the previous sections carry over to the average reward case.
In this section we do not make the assumption that $\Lambda+\beta=1$, since we consider sequences of values of $\beta$, keeping the remaining parameters fixed.

For each policy $\pi\in\Pi$, the long-run average expected net profit
given that the initial state is $(x,i)\in S$ is
\begin{equation}\label{avereward}
g^{\pi}(x,i)=\lim\sup_{t\rightarrow\infty}\frac{1}{t}\upsilon^{\pi}_t(x,i),\;(x,i)\in
S,\end{equation} where

\begin{eqnarray*} \lefteqn{\upsilon^{\pi}_{t}(x,i)=}\\ & &\Ex^{\pi}
\left [ \sum_{j=0}^{N_t} R\ {\bf{1}}(\Aad(T_j)=1)\right.\\ & &\left.-
\int_{0}^{t}[h(X(u))+c(\Aser(u))]du | X(0)=x, I(0)=i
\vphantom{\sum_{j=0}^{v_t}} \right ]
\end{eqnarray*}
denotes the expected net profit generated by the process $(X(t),I(t))$
up to time $t$ with initial state $(x,i)\in S$ and $N_t$ the number of
admission decisions made up to time $t$.  The optimal average expected
net profit is defined as
\begin{equation}\label{optavreward}
g^*(x,i)=\sup_{\pi\in\Pi}g^{\pi}(x,i),\;(x,i)\in S.\end{equation}

A policy $\pi^*$ is characterized as  average-reward optimal
 if $g^{\pi^*}(x,i)=g^*(x,i)$ for all $(x,i)\in S$.

As in Section \ref{subsec-DTMDP}, we transform the problem into an
equivalent model in discrete time using uniformization, where the
expected time between decision epochs is equal to
$\frac{1}{\Lambda}$. The resulting discrete time Markov decision process is
described by the following average reward optimality inequalities
\begin{eqnarray}
w(x,1)&\leq& \max \left \{ R + w(x+1,0), w(x,0)\right \},\;x\geq0\label{AROI1} \\ \nonumber
w(x,0)&\leq& -\frac{h(x)}{\Lambda} - \frac{g}{\Lambda}   +
\frac{\lambda}{\Lambda}  w(x,1)   \\
&&  + \max_{j=1,\ldots,k} \{ -\frac{c}{\Lambda}+\frac{\mu_j}{\Lambda} w(x-1,0) + \frac{\mu_k-\mu_j}{\Lambda} w(x,0)\}
,\;x>0\label{AROI2}\\
w(0,0)&\leq& -\frac{h(x)}{\Lambda}-\frac{g}{\Lambda} -\frac{c_1}{\Lambda}+
\frac{\lambda}{\Lambda} w(0,1) + \frac{\mu_k}{\Lambda} w(0,0) .\label{AROI3}
\end{eqnarray}
with respect to a constant $g$ and a real-valued function $w$ on $S$.
These correspond to  the  discrete-time discounted optimality equations \eqref{DTMDP1}-\eqref{DTMDP-int}.
Note that as the continuous discount rate $\beta \downarrow 0$, the equivalent discrete-time discount factor $\frac{\Lambda}{\Lambda+\beta} \uparrow 1$.

Theorem 7.2.3. of \cite{sennott}, provides a set of sufficient
conditions (SEN assumptions, p.135) so that: (a) a solution $(w,g)$ to
\eqref{AROI1}-\eqref{AROI3} exists, (b) a
long run average reward optimal policy exists, which realizes the
maximum in \eqref{AROI1}-\eqref{AROI3} and  is obtained as the limit
of a sequence of discounted optimal policies under a subsequence of
discount rates $\beta_n\downarrow0$, (c) $g$ is equal to the optimal
average net profit and  is obtained as a limit of the optimal
discounted expected net profit for $\beta\downarrow0$ and (d)
$w(x,i)$ is a limit function of the sequence
$w_{\beta_n}=\upsilon_{\beta_n}(x,i)-\upsilon_{\beta_n}(0,0)$ with
$\beta_n\downarrow0$.

In the following theorem we prove that Assumption 1, which was shown
in Theorem \ref{VerificationTheorem} to be sufficient for the
existence of a solution to the discounted problem, also ensures that
the average reward problem has an optimal solution.

\begin{theorem}\label{estaveropt}
If the holding cost rate function satisfies Assumption 1, then
\begin{itemize}
\item[{\bf{i.}}] There exists a stationary long-run average reward optimal
policy $\pi^*$, which is a limit point of a sequence of stationary
discounted expected net profit optimal policies, i.e.,
$$
\pi^*=\lim_{n\rightarrow\infty}\pi_{\beta_n},
$$
where $\{\beta_n, n \geq 1\}$ is any sequence of
discount rates such that $\beta_n\downarrow0$ and
$\pi_{\beta_n}$ is a $\beta_n$-discount optimal stationary policy.

\item[{\bf{ii.}}]
The expected average net profit associated with
$\pi^*$ is equal to
\begin{equation}
\label{optavrew}
g^*=\lim_{\beta\downarrow0}\beta\upsilon_{\beta}(x,i),
\end{equation} for every $(x,i)\in S$.

\item[{\bf{iii.}}]
For any sequence  $\beta_n\downarrow0$ in (i), the sequence of functions $\{w_{\beta_n}, n\geq 1\}$ defined by
$$
w_{\beta_n}(x,i)=\upsilon_{\beta_n}(x,i)-\upsilon_{\beta_n}(0,0)
$$
converges pointwise to a function $w$ such that
$(w,g^*)$ satisfy   \eqref{AROI1}-\eqref{AROI3}.

\end{itemize}
\end{theorem}

\begin{proof} From Theorem 7.2.3. of \cite{sennott} it is sufficient
to verify the SEN assumptions.
In our problem the  state space $S$ is countable
and from Theorem \ref{VerificationTheorem} and Lemma \ref{burdenprp1}
the value function $\upsilon(x,i)$ of the discounted problem is
nonincreasing in $x\in\mbox{I\!N}_0$ for any $i\in\{0,1\}$. Thus,  we can
apply Corollary 7.5.4 of \cite{sennott}, which states that a
sufficient condition for SEN assumptions to hold in this case is the
existence of a $0$-standard policy, i.e. a (generally randomized) policy $d$ which
induces an irreducible and positive recurrent Markov process with finite
expected first passage time  from any state $s$ to state ${\bf{0}}$,
$m_{s{\bf{0}}} < \infty, s \in S$ and
expected first passage profit from any state $s$ to state ${\bf{0}}$,
$w_{s{\bf{0}}}>-\infty,\;x>0$.

To show the existence of a 0-standard policy for any $p \in [0,1]$,
let $d(p)$ be the randomized policy under which the service rate is
always set to $\mu_1$ and arriving customers are admitted with
probability $p$. If $p$ is such that $\frac{\lambda p}{\mu_1}<1$, then
under policy $d(p)$ the system is equivalent to a stable $M/M/1$ queue
$\{X_t , t\geq 0\}$, with state $X_t$ denoting the number of customers
in the system, arrival rate $\lambda p$, service rate $\mu_1$,
admission reward $R$ and cost rate $h(x)$ while in state $x$.

The discrete-time equivalent of this process corresponds to a positive recurrent Markov chain
$\{X_n, n=0,1,\ldots\}$, with transition probabilities
$$
P_{x, x+1} =\frac{\lambda p}{\Lambda},
P_{x, x-1} =\frac{\mu_1}{\Lambda},
P_{x, x} =1- \frac{\lambda p +\mu_1}{\Lambda}, x>0,
$$
$$
P_{01}=1-P_{00}= \frac{\lambda p}{\Lambda},
$$
and one-step rewards $r(x)=\frac{R \lambda p -h(x)}{\Lambda}$.

Let $N=\min\{n > 0 : X_n=0\}$ denote the first passage time to state
0. It is well known that the expected first passage times to state 0
in the continuous-time M/M/1 queue are equal to $\frac{x}{\mu_1 -
  \lambda p}, x>0$, thus in the discrete time model
$$
m_{x0} \equiv E(N|X_0=x) =  \frac{\Lambda x}{\mu_1 - \lambda p} < \infty.
$$

Now consider the first passage expected  profit
$$
w_{x0} = E \left ( \sum_{n=0}^T r(X_n) | \ X(0)=x \right ) =
\frac{R \lambda p}{\Lambda} m_{x0} - \frac{1}{\Lambda} E \left ( \sum_{n=0}^T h(X_n) | \ X(0)=x \right ) .
$$

Thus, to show that $w_{x0} > -\infty$, it suffices to show that the expected
first passage holding cost is finite, i.e.,
$$
H_{x} \equiv E \left ( \sum_{n=0}^T h(X_n) | \ X(0)=x \right ) < \infty, \ \ x > 0.
$$

From Corollary C.2.4 of \cite{sennott} it follows that in order to show $H_x < \infty$, it is sufficient to establish that there exists a nonnegative finite function $W(x)$ such that
\begin{eqnarray}
  \label{eq:wbound0}
  \sum_jP_{0j}W(j) &<& \infty,\\
  \label{eq:wboundx}
   \sum_j P_{xj} (W(x) - W(j)) &\geq& h(x),\ \  x>0.
\end{eqnarray}

We will prove that there exist a sufficiently small $p$ and a
sufficiently large $M>0$ such that these inequalities are satisfied by
function $W(x) = M \theta^x$, where $\theta>1$ is the constant
appearing in Assumption 1. First, \eqref{eq:wbound0} is immediate,
since $\sum_jP_{0j}W(j) = M \frac{\lambda p}{\Lambda} \theta$.  For
$x>0$,
\[
   \sum_j P_{xj} (W(x) - W(j))=
\frac{\lambda p}{\Lambda} M(\theta^{x}-\theta^{x+1}) + \frac{\mu_1}{\Lambda} M(\theta^x - \theta^{x-1}) \\
=
M \theta^{x-1} (\theta -1) \frac{\mu_1 - \lambda p \theta}{\Lambda}.
\]
On the other hand, from (c1) of Assumption 1, it follows that $h(x) \leq h(1) \theta^{x-1}, x>0$. Therefore,
if we take $p < \min(\frac{\mu_1}{\lambda p}, 1)$, so that $\mu_l - \lambda p >0$ and
$M\geq \frac{h(1) \Lambda}{(\mu_1 - \lambda p)(\theta-1)}$, then it is true that
\[
\sum_j P_{xj} (W(x) - W(j)) \geq h(1) \theta^{x-1} \geq h(x), \ \ x>0.
\]

Summarizing, we have shown that for sufficiently small $p$ there
exists a function $W(x)$ satisfying \eqref{eq:wbound0},
\eqref{eq:wboundx}, thus  policy $d(p)$ is 0 standard  and
the proof of the theorem is complete.
\end{proof}

From Corollary 7.5.4 of \cite{sennott}, we also obtain that the
relative value function $w(x,i)$ is nonnegative and nonincreasing in
$x$ for any $i$.

The results on the structure of the discounted optimal policy are
extended to the average reward case, since the average optimal policy
is obtained as a limit of a sequence of discounted optimal policies and the action set is finite.

Specifically, from Theorem \ref{estaveropt} there exist admission and service thresholds in both problems which are obtained as limits for any sequence $\beta_n\downarrow0$ of the corresponding threshold values in the infinite horizon discounted versions of the joint control problem and the admission control subproblem. For the following, we let $B^{s,av}_j$ be the service thresholds for $j=1,\ldots,k$ and $B^{d,av}$ be the admission threshold of the joint problem under average reward criterion, as well as $\hat{B}^{d,av}$ be the corresponding admission threshold in the admission control subproblem under the average reward criterion.

Regarding the value of service flexibility, since the optimal average reward does not depend on the initial state, we can obtain more concrete results. Specifically, we define the service rate flexibility as $\hat{\epsilon}=g-\hat{g}$, i.e. the effect on the average reward, where $g,\;\hat{g}$ denote the optimal average reward of the joint problem and the admission control subproblem, respectively.

We can now prove the following result, which is the analog of Theorem \ref{vof_novalue} in the average case.

\begin{theorem}\label{vof_av}
If $R\leq\zeta_2$, then $\hat{B}^{d,av}=B^{d,av}$, and $\hat{\epsilon}=0$.
\end{theorem}

\begin{proof}
Consider any sequence $\beta_n\downarrow0$ and we denote explicitly the dependence of the discounted joint problem thresholds $B^{s}_j(\beta_n),\;B^{d}(\beta_n)$, as well as of the admission threshold  $\hat{B}^{d}(\beta_n)$ in the admission control subproblem.

If $R\leq\zeta_2$, it follows from Theorem \ref{vof_novalue} that for any $n$
\begin{equation}B^{d}(\beta_n)=\hat{B}^{d}(\beta_n)\label{admthr_bn}\end{equation}
and \begin{equation}\label{vof_bn}\upsilon_{\beta_n}(x,i)=\hat{\upsilon}_{\beta_n}(x,i) \mbox{ for any } x=0,1,\ldots,B^{s}_1(\beta_n),\;i=0,1,\end{equation}
since $\hat{\epsilon}(x,i)=0$ for any $x=0,1,\ldots,B^{s}_1(\beta_n)$ and $i=0,1$.

Taking limits in \eqref{admthr_bn} with $n\rightarrow\infty$, we derive that $\hat{B}^{d,av}=B^{d,av}.$

Also, since $B^s_1(\beta_n)\geq0$, it follows from \eqref{vof_bn} that $\upsilon_{\beta_n}(0,0)=\hat{\upsilon}_{\beta_n}(0,0)$ for any $n$, thus, from Theorem \ref{estaveropt} \it{ii.},
\begin{equation} \lim_{n\rightarrow\infty}\beta_n \upsilon_{\beta_n}(0,0)= \lim_{n\rightarrow\infty}\beta_n \hat{\upsilon}_{\beta_n}(0,0).\label{limbn}\end{equation}
Since \eqref{optavrew} holds for every state $(x,i)$, thus also for (0,0), it follows from \eqref{limbn} that $\hat{\epsilon}=0$.
\end{proof}

\section{Reward Collected at Departure Epochs}
\label{Rdepart}

In the previous sections it was assumed that the reward $R$ is
collected upon admitting a customer. While this is plausible in many
situations such as ticket-based operations or call centers with
upfront charge, it is also often the case that the service reward is
collected at the time of departure of the customer, for example in
jobshops with payment upon delivery. In this section we formulate the
corresponding MDP model for the second case, point out the
similarities and differences and show that essentially all the
conclusions obtained so far still hold.

When $R$ is collected at departure epochs, the Markov Decision Process
in finite horizon corresponding to \eqref{DTMDP1}-\eqref{DTMDP-int} now takes the following form:
\begin{eqnarray}
\upsilon_{n+1}(x,1)&=& \max\{\upsilon_{n+1}(x+1,0),\upsilon_{n+1}(x,0)\},\;x\geq0,\label{DTMDP1d} \\
\nonumber\upsilon_{n+1}(x,0)&=&-h(x)+\lambda\upsilon_{n}(x,1)\\
&&+\max_{j=1,\dots,k}\{-c_j+\mu_j(R+\upsilon_{n}(x-1,0))+(\mu_k-\mu_j)\upsilon_{n}(x,0)\},\;x>0,\label{DTMDP2d}\\
\upsilon_{n+1}(0,0)&=&-h(0)-c_1+\lambda\upsilon_{n}(0,1)+\mu_k\upsilon_{n}(0,0),\label{DTMDP3d}\\
\upsilon_{0}(x,i)&=&0,\;x\geq0,\;i\in\{0,1\}.\label{DTMDP-intd}
\end{eqnarray}
Note that in \eqref{DTMDP1d}-\eqref{DTMDP-intd} the term for $R$ is
added at transitions from $(x,0)$ to $(x-1,0)$, whereas at admission
epochs there is no reward collected.

As in the original model, let $\Delta_n(x,i)=\upsilon_n(x,i)-\upsilon_n(x+1,i)$.

One and, in essence, the only difference between the two models is
that Lemma \ref{burdenprp1} is not true anymore, i.e., the value
function is not nonincreasing in $x$. This is intuitively expected,
since an additional customer in the queue brings with him the prospect
of a future reward as well as a burden due to the higher holding
costs. Mathematically, the induction proof of Lemma \ref{burdenprp1}
changes at state $(0,0)$.  Indeed, it is now true that
$$
\Delta_{n+1}(0,0)=\Delta h(0) +
\lambda\Delta_{n}(0,1)+\min_{j=1,\ldots,k}\{c_j-c_1-\mu_j R+(\mu_k-\mu_j)\Delta_{n}(0,1)\},
$$
from which it follows that $\Delta_{n+1}(0,0)\geq -R$, which can be generalized by induction, i.e. $\Delta_{n+1}(x,0)\geq -R$ for any $x$.

Since it is not generally true that $\Delta_n(x,i)\geq 0$, the difference
cannot be interpreted as a burden, but rather as the net effect of an
additional customer, which can be either a burden or a benefit.

On the other hand, by following the remaining proofs in the original
model, it can be verified that all the results on the monotonicity of
$\Delta_n(x,0)$ in $x$ and the threshold structure of the optimal policy still hold.

The admission and service rate thresholds now take the form
\begin{equation*}
          B^s_{n,j}=1+T_{\Delta_n(\cdot,0)}(\zeta_{j+1}-R),\mbox{ for }j=0,\ldots,k \ \
B^d_{n}=T_{\Delta_n(\cdot,0)}(0),
\end{equation*}
where $\Delta_n$ corresponds to the new net effect function and
$T_f(\theta)$ is the same generic threshold function defined in
\eqref{thresgenfcn}.

Moreover, the same value of $j_0$, as in Section \ref{subsec-threspol}, i.e. $j_0=\min\{j\geq1:\; R\leq\zeta_{j+1}\}$, satisfies the analog of Proposition \ref{optpol} for the new thresholds. This implies that the distinction between useful and useless service rates does not depend on whether the payment is made before or after service.

Regarding the results on the value of flexibility, with the given changes in the optimality equations of the joint control problem, which are similar in the corresponding control subproblem, all the proofs still hold. Specifically, letting $\hat{\epsilon}_n (x,i)=\upsilon_n (x,i)-\hat{\upsilon}_n (x,i)$, for any $(x,i)$ be the value of service flexibility, as in Section \ref{sec-VoSF}, the following results hold.
\begin{enumerate}
\item[i.] $\hat{\epsilon}_n (x,i)$ is nondecreasing in $x$ for all $i$, for any $n$.
\item[ii.] $\hat{B}^d_n\leq B^d_n$ for any $n$.
\item[iii.] If $R\leq\zeta_2$, then $\hat{B}^d\leq B^d$ and $\hat{\epsilon}(x,i)=0$ for $i=0,1$ and $x=0,1,\ldots,B^s_1$, where $B^s_1, B^d, \hat{B}^d$ are the service and the admission thresholds in the infinite horizon discounted version of the joint and the admission control subproblem, and $\hat{\epsilon}(x,i)$ refers to the value of flexibility derived by the infinite horizon counterparts of the corresponding value functions.
\end{enumerate}

\section{Computational Results}
\label{sec-Numerical}
In this section we present the results of some computational
experiments, which explore the value of service rate flexibility and
the influence of the service rate switch option on the optimal
admission policy for the discounted profit maximization problem with service payment before and after service and the average profit
maximization problem with service payment upon customers admission.

We first consider the behavior of the optimal policy, the comparison of admission thresholds and the value of service flexibility as a function of the service reward $R$. We consider a system with input rate $\lambda=10$, available service rates $\mu_1=4,\;\mu_2=8,\;\mu_3=12$ and $\mu_4=16$
with corresponding cost rates $c_1=0,\;c_2=2,\;c_3=6\;,c_4=14$, holding cost rate $h(x)=0.1 x^2$ discount rate $\beta=1$. The corresponding values of the relative increase in cost rate with respect to the service rate are $\zeta_1=0.5,\;\zeta_2=1,\;\zeta_3=1.5$ and $\zeta_4=2$. For the computations we truncate the state space up to $x=20$. Note that the parameters do not
generally satisfy the normalization assumption, $\Lambda+\beta=1$.
However, the required rescaling has been performed in the computations.

These results are presented in
Figures \ref{serflexdcbef}, \ref{serflexdcaft} and \ref{serflexavbef}. Figures \ref{serflexdcbef} and \ref{serflexdcaft} correspond to the discounted profit case with payment before and after service, respectively, whereas Figure \ref{serflexavbef} to the average reward case. In these diagrams, we observe that the service
thresholds $B_j^s$ are nonincreasing in $R$ ($B^s_4$ appears constant at $20$ because of the state space truncation), while both admission thresholds
$B^d$ and $\hat{B}^d$ are nondecreasing. This behavior is intuitive because as $R$ increases admitting customers becomes more profitable and the reward to cost relationship motivates higher service rates to be employed more often. Moreover, the threshold diagrams in all cases are consistent with the results of Proposition \ref{optpol}, i.e. when $R\leq\zeta_{i+1}$ then $B_i^s\geq B^d+1$.

Considering the value of flexibility in each case, we observe that it is nondecreasing in $R$, which is also intuitive. Furthermore, in the previous sections it was established
analytically that when $R\leq\zeta_2$ the service rate
flexibility is essentially of no value. This was manifested in three different ways. Specifically, it was proved that $R\leq\zeta_2$ implies
$B_1^s\geq B^d+1$, i.e. that all service rates higher than $\mu_1$ are essentially not used. This in turn implies that the admission thresholds in the
joint and the admission control problem coincide, and that the value of service flexibility is equal to zero.

However, even if $R>\zeta_2$, it may still be true that $B_1^s\geq B^d+1$, and even if $B_1^s< B^d+1$ it may also be true that
$\hat{B}^{d}=B^d$. Indeed, we observe in Figures \ref{serflexdcbef} to \ref{serflexavbef} that there exists a small interval of values of $R>\zeta_2$ such that all the above occur.

Thus, it is natural to ask to what extent the condition $R\leq\zeta_2$ is a good indicator for zero value of the service flexibility. In order to do this, for varying of $\zeta_2$ we can numerically identify two critical values $\tilde{R}^s$ and $\tilde{R}^d$ as follows. $\tilde{R}^s$ is the minimum value of $R$ which makes $B_1^s< B^d+1$ and $\tilde{R}^d$ is the minimum value of $R$ which makes $\hat{B}^{d}< B^d$. From the previous discussion it follows that $\zeta_2\leq\tilde{R}^s\leq\tilde{R}^d$. The condition $R>\zeta_2$ is a good approximation to characterize the value of service flexibility to the extent that these values are close to each other.

To examine this issue, we perform another numerical
experiment for the discounted case with payment before service completion. We consider only two available service rates $\mu_1=3,\;\mu_2=5$, with arrival rate $\lambda=5$, holding cost $h(x)=x^2$ and $\beta=0.5$. We vary $\zeta_2$ from $0$ to $5$, by varying the cost rate $c_2$ keeping $c_1=0$. In Figure \ref{threscombdg}, we present the curves of $\tilde{R}^s$ and $\tilde{R}^d$ as functions of $\zeta_2$. The curve of $\tilde{R}^s$ is indeed above the
diagonal $R=\zeta_2$, however only slightly, whereas $\tilde{R}^d$ may be significantly higher than $\zeta_2$ and $\tilde{R}^s$. This shows that
the condition $R>\zeta_2$ is a good approximation in order to show when a higher service rate is employed whereas this is not the case to
characterize the effect of service rate switch to a higher value on admission thresholds. Thus, there exist cases where the system manager is motivated to switch to a higher service rate even though he accepts the same number of customers, in order to decrease the holding costs.

We finally consider the interaction between the timing of service payment and the value of service flexibility. In Figure \ref{serflexmhdg} we present the optimal profits of the joint control discounted models where service payment is made before and after service completion, as well as the values of the corresponding service flexibilities for the same parameter values as in the first experiment. When the service revenue is collected at customer departure, its effective value upon admission is discounted by the sojourn time of the customer in the system, thus the optimal profit is lower when payment is made on customer departure. Therefore, in this case the service rate switch option is even more useful because, in addition to decreasing the holding costs, it also increases the value of the service payment by reducing the sojourn time. Indeed, the service rate flexibility is higher for payment after the service completion.

\begin{figure}
\centering
\includegraphics[angle=0, width=0.8\textwidth]{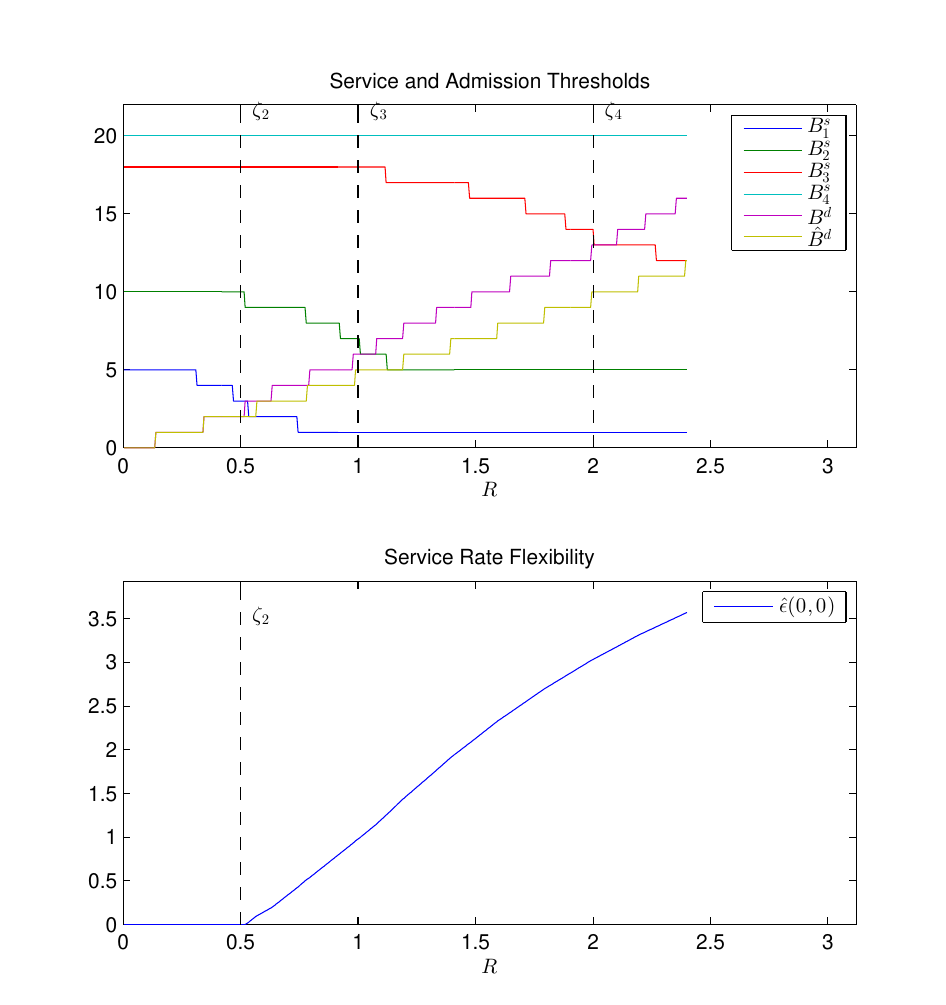}
\caption{The admission and service thresholds and the value of service flexibility as a function of $R$ in the discounted case when service payment is made before service completion for $\lambda=10,\;\mu_1=4,\mu_2=8,\mu_3=12,\mu_4=16\;c_1=0,c_2=2,c_3=6,c_4=14,\;h(x)=0.1x^2,\;\beta=1$} \label{serflexdcbef}
\end{figure}

\begin{figure}[H]
\centering
\includegraphics[angle=0, width=0.8\textwidth]{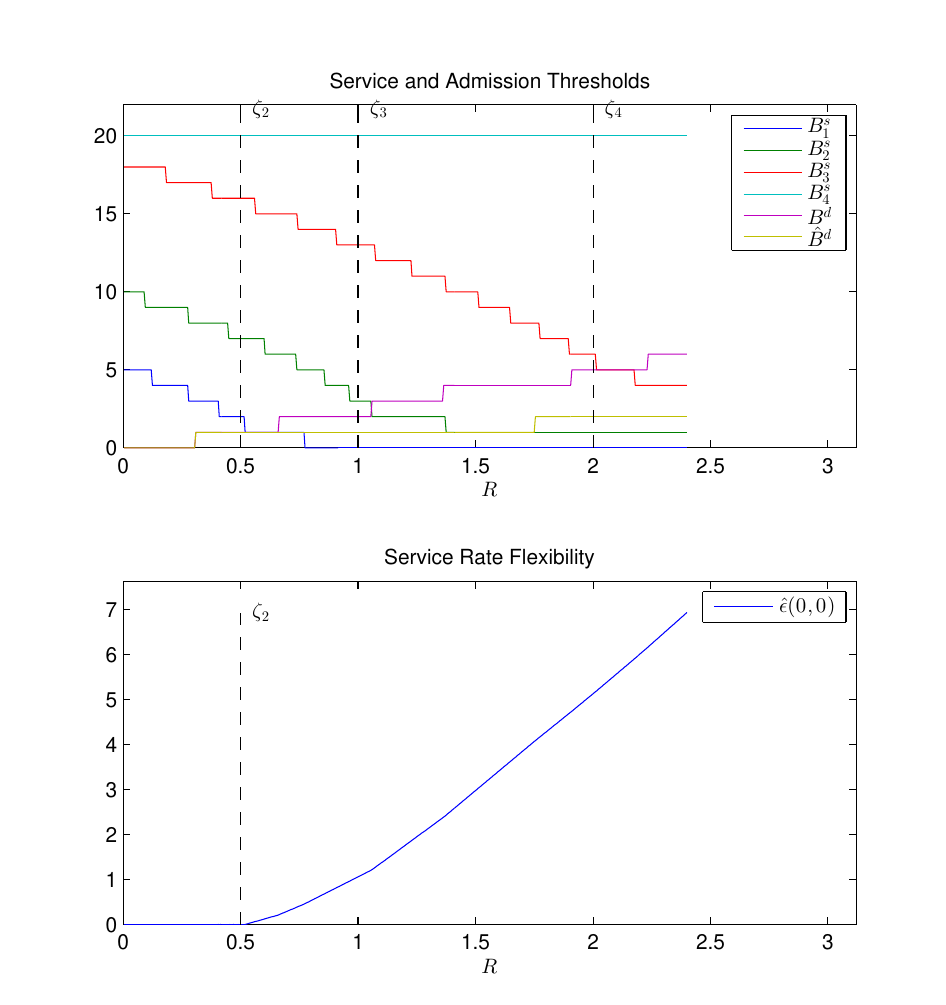}
\caption{The admission and service thresholds and the value of service flexibility as a function of $R$ in the discounted case when service payment is made after service completion for $\lambda=10,\;\mu_1=4,\mu_2=8,\mu_3=12,\mu_4=16\;c_1=0,c_2=2,c_3=6,c_4=14,\;h(x)=0.1x^2,\;\beta=1$} \label{serflexdcaft}
\end{figure}

\begin{figure}[H]
\centering
\includegraphics[angle=0, width=0.8\textwidth]{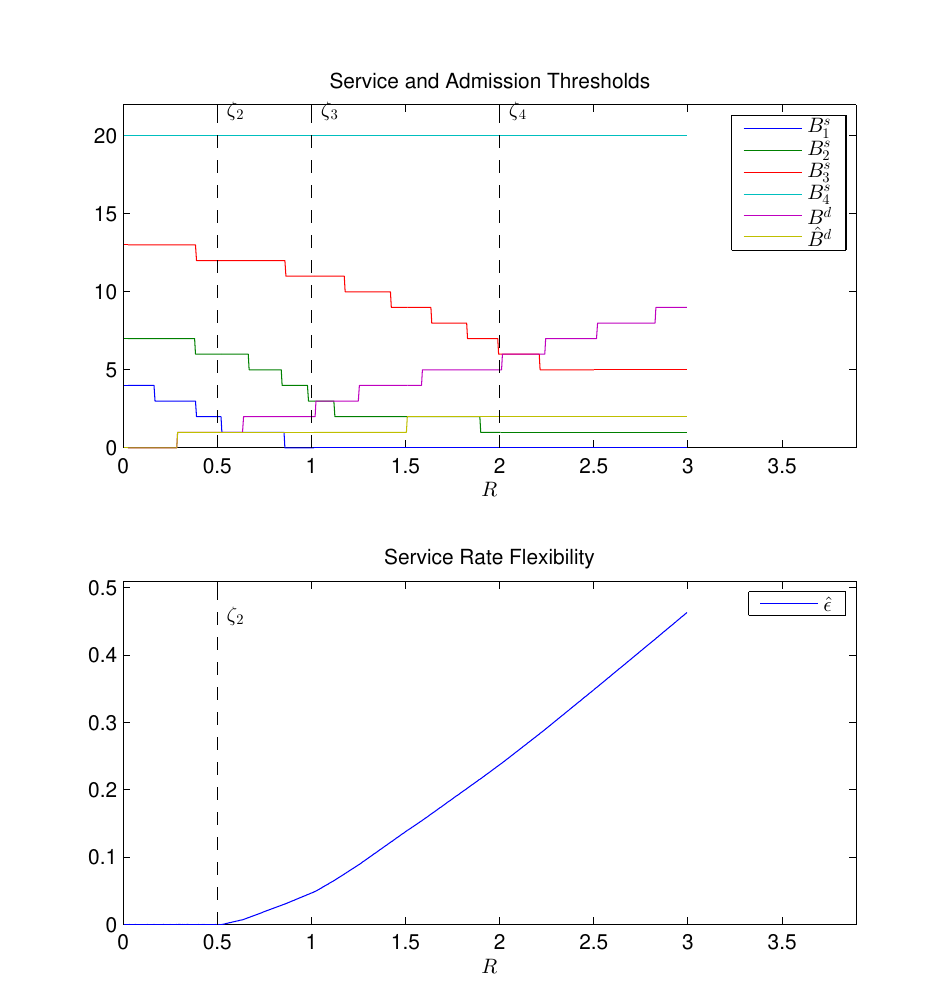}
\caption{The admission and service thresholds and the value of service flexibility as a function of $R$ in the average case when service payment is made before service completion for $\lambda=10,\;\mu_1=4,\mu_2=8,\mu_3=12,\mu_4=16\;c_1=0,c_2=2,c_3=6,c_4=14,\;h(x)=0.1x^2$} \label{serflexavbef}
\end{figure}

\begin{figure}[H]
\centering
\includegraphics[width=0.8\textwidth]{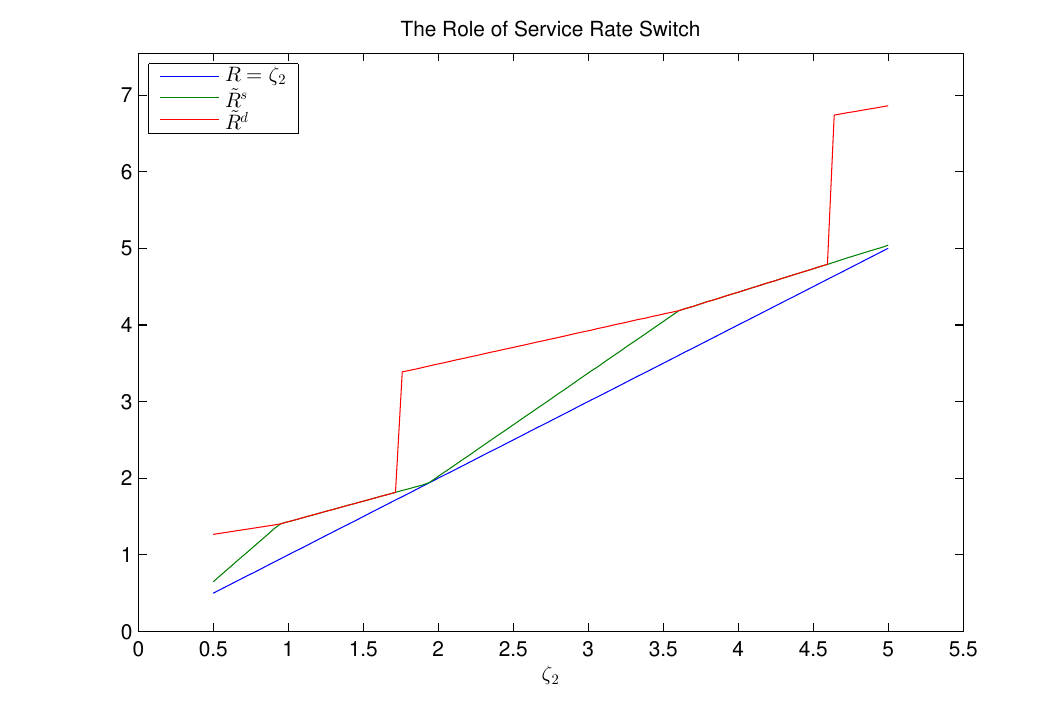}
\caption{Critical values $\tilde{R}^s$ and $\tilde{R}^d$ as a function of $\zeta_2$ for
  $\lambda=5,\;\mu_l=3,\;\mu_2=5,\;h(x)=x^2,\;\beta=0.5$}
\label{threscombdg}
\end{figure}

\begin{figure}[H]
\centering
\includegraphics[width=0.8\textwidth]{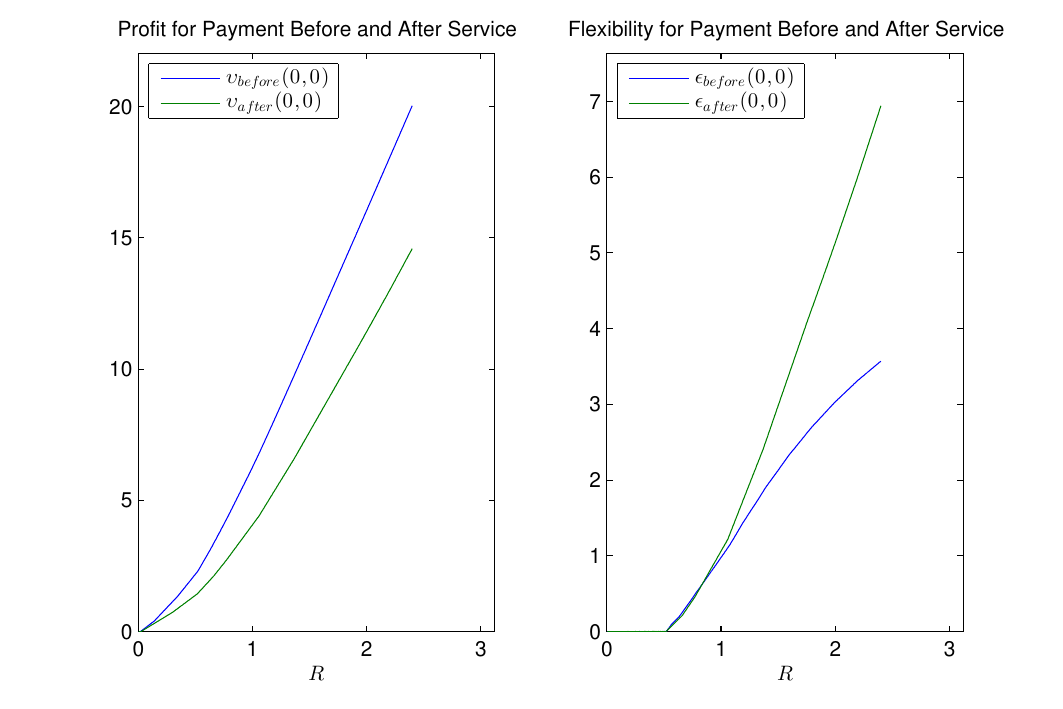}
\caption{Optimal profits and value of service flexibility as a function of $R$ in the discounted case for the joint control models
 where service payment is made before and after service completion for
  $\lambda=10,\;\mu_1=4,\mu_2=8,\mu_3=12,\mu_4=16\;c_1=0,c_2=2,c_3=6,c_4=14,\;h(x)=0.1x^2,\;\beta=1$} \label{serflexmhdg}
\end{figure}

\section{Conclusions and Extensions}
\label{sec-Conclusions} In this paper we analyzed the problem of joint
dynamic admission and service control in an $M/M/1$ queue under
expected discounted and average profit maximization.  We established a threshold
structure for the optimal service rate-admission control policy.  We
defined the value of the service rate flexibility as the benefit that
the option to switch to a higher service rate brings to a system with pure admission
control, and showed that the value of flexibility is nondecreasing with
system congestion.  We finally identified a simple condition between
the admission reward and the relative costs of high service rates, under
which the admission policy is not affected and the value of service
flexibility is zero in low congestion states for the discounted problem. For the average reward problem this leads to
a stronger result, that under the same condition, the option to switch to a higher service rate does not affect the average profit.
\\
The analysis in this paper was conducted in the context of an $M/M/1$
queue, where a higher service rate means that the server operates at
higher speed. However in many real applications such as banks or call
centers the service capacity is affected by dynamically varying the
number of operating servers. This corresponds to an $M/M/m$ queue with
$m$ being determined by a dynamic policy.  We conjecture that the
admission policy will still be threshold-based. On the other hand, it
would be interesting to study how the service rate policy is
structured under different assumptions on the service rate switch
option, e.g. assuming that all servers are required to use the same
service rate, or that each server is allowed to select its own rate.
\\
Furthermore, the assumption of identical customers may be relaxed by assuming multiple
customer classes differentiated by admission reward and/or service
rate. If the service rates are identical among classes, it is expected
that the optimal admission policy is determined by class-dependent
admission thresholds which are increased when a higher service rate
option becomes available.
\\
It would also be interesting to consider
the value of service rate flexibility when there is a fixed service
rate switch cost, which brings hysteretic policies into play. However, the inclusion of a switching cost increases the complexity
of the problem significantly, because the state variable must include the service rate currently employed. Since we consider admission policies as well,
the framework of conjectured optimal policies may now be much wider, because in principle the admission threshold could depend on the current service rate.
\\
Finally, in the present model we considered maximization of net profit
from the point of view of a single decision maker, who may be the
system owner or a collective representative maximizing the total
customer benefit. It would be interesting to consider equilibrium
strategies in a game theoretic model where the server determines the
service rate and arriving customers respond by deciding individually
whether to join the queue or balk.  Such a model could also include
pricing as an additional policy component available to the server.

\centerline{{\bf Acknowledgment}}
This research was supported by the Hellenic Secretariat of Research and Technology via a Greece-Turkey bilateral research program.

\bibliographystyle{abbrvnat}

\end{document}